\documentclass[oneside, 12pt]{article}
\usepackage[margin=3cm]{geometry}

\usepackage[utf8]{inputenc}
\usepackage[english]{babel}

\usepackage{graphicx}
\usepackage{amsthm}   
\usepackage{amsmath, xspace}  
\usepackage{amssymb}  
\usepackage{mathabx} 
\usepackage{bm} 
\usepackage{verbatim} 
\usepackage{enumitem} 
\usepackage{cancel}   
\usepackage{multicol}
\usepackage[dvipsnames]{xcolor} 

\usepackage{hyperref}
\hypersetup{
  colorlinks=true,
  citecolor=black,
  linkcolor=black,
  urlcolor=black,
  breaklinks=true,
  bookmarksnumbered=true,
  bookmarksopen=true,
  pdftitle={An Escape from Vardanyan's Theorem},
  pdfauthor={Ana de Almeida Borges and Joost J. Joosten},
  pdfsubject={},
  pdfcreator={Ana de Almeida Borges},
  pdfkeywords={}
}
\usepackage[numbers]{natbib}   

\theoremstyle{definition}
\newtheorem{theorem}{Theorem}[section]
\newtheorem{definition}[theorem]{Definition}
\newtheorem{lemma}[theorem]{Lemma}

\newtheorem{remark}[theorem]{Remark}

\newcounter{dummy}
\makeatletter
\newcommand\myitem[1][]{\item\refstepcounter{dummy}\def\@currentlabel{#1}}
\makeatother

\makeatletter
\newcommand{\inspect}[1]{%
  \def\inspectspace{\mskip3mu\relax}
  \sbox\z@{$#1$}
  \sbox\tw@{\thickmuskip=0mu$#1$}
  \ifdim\wd\tw@<\wd\z@ \def\inspectspace{\mskip9mu\relax}\fi 
}
\makeatother

\DeclareMathOperator{\Existsa}{\exists}
\newcommand{\Exists}[1]{
  \inspect{#1}
  \Existsa #1 \inspectspace
}

\DeclareMathOperator{\Foralla}{\forall}
\newcommand{\Forall}[1]{
  \inspect{#1}
  \Foralla #1 \inspectspace
}

\newcommand{\Lbf}{\mathcal{L}_{\Box, \forall}}

\newcommand{\F}{\mathcal{F}}
\newcommand{\M}{\mathcal{M}}

\newcommand{\la}{\langle}
\newcommand{\ra}{\rangle}

\newcommand{\N}{\mathbb{N}}
\newcommand{\PA}{\mathsf{PA}}
\newcommand{\HA}{\mathsf{HA}}

\newcommand{\gl}{\mathsf{GL}}
\newcommand{\GLP}{\mathsf{GLP}}
\newcommand{\RC}{\mathsf{RC}}
\newcommand{\QRC}{\mathsf{QRC_1}}
\newcommand{\QGL}{\mathsf{QGL}}
\newcommand{\Kf}{\mathsf{K}4}
\newcommand{\QKf}{\mathsf{QK}4}
\renewcommand{\S}{\mathsf{S}}
\newcommand{\QS}{\mathsf{QS}}

\newcommand{\isig}[1]{{\ensuremath {\mathrm{I}\Sigma_{#1}}}\xspace}

\newcommand{\PL}{\mathsf{PL}}
\newcommand{\TPL}{\mathsf{TPL}}
\newcommand{\QPL}{\mathsf{QPL}}
\newcommand{\TQPL}{\mathsf{TQPL}}

\newcommand{\QRCt}{\mathcal{QRC}_1}

\newcommand{\Ax}{\normalfont{\text{Ax}}}

\newcommand{\Prf}{\normalfont{\text{Prf}}}

\newcommand{\fv}{\normalfont{\text{fv}}} 
\newcommand{\mdepth}{\normalfont{\text{d}}_\Diamond} 
\newcommand{\udepth}{\normalfont{\text{d}}_\forall} 
\newcommand{\cdepth}{\normalfont{\text{d}}_{\normalfont{\text{const}}}} 
\newcommand{\Cl}{\mathcal{C}\ell} 

\renewcommand{\vec}[1]{\bm{#1}}

\newcommand{\gnum}[1]{\ulcorner #1 \urcorner} 

\newcommand{\xaltern}[1]{\sim_{#1}} 

\newcommand{\subst}[2]{[#1 {\leftarrow} #2]} 
\newcommand{\pair}[2]{\la #1, #2 \ra}

\newcommand{\umod}{\  \normalfont{\text{mod}} \ }

\newcommand{\fs}{^{*}} 
\newcommand{\fsT}{^{*_{T}}} 
\newcommand{\fsPA}{^{*_{\PA}}} 

\newcommand{\ofs}{^{\circledast}} 
\newcommand{\ofsT}{^{\circledast_{T}}} 
\newcommand{\ofsPA}{^{\circledast_{\PA}}} 

\newcommand{\is}{^{\circ}} 
\newcommand{\isT}{^{\circ_{\tau}}} 
\newcommand{\isHA}{^{\circ_{\eta}}} 

\newcommand{\ois}{^{\circledcirc}} 
\newcommand{\oisT}{^{\circledcirc_{\tau}}} 

\newcommand{\bij}[1]{\lcorners #1 \rcorners} 

  \title{An Escape from Vardanyan's Theorem}
  \author{Ana de Almeida Borges\thanks{\url{anadealmeidagabriel@ub.edu}} \and Joost J. Joosten\thanks{\url{jjoosten@ub.edu}}}
  \date{Universitat de Barcelona}
  
\begin{document}

\maketitle

  \begin{abstract}
    Vardanyan's Theorems \cite{Vardanyan1986, Vardanyan1988} state that $\QPL(\PA)$ -- the quantified provability logic of Peano Arithmetic -- is $\Pi^0_2$ complete, and in particular that this already holds when the language is restricted to a single unary predicate. Moreover, Visser and de Jonge \cite{VisserAndDeJonge:2006:NoEscape} generalized this result to conclude that it is impossible to computably axiomatize the quantified provability logic of a wide class of theories.
    However, the proof of this fact cannot be performed in a strictly positive signature.
    The system $\QRC$ was previously introduced by the authors \cite{QRC1} as a candidate first-order provability logic.
    Here we generalize the previously available Kripke soundness and completeness proofs, obtaining constant domain completeness.
    Then we show that $\QRC$ is indeed complete with respect to arithmetical semantics. This is achieved via a Solovay-type construction applied to constant domain Kripke models.
    As corollaries, we see that $\QRC$ is the strictly positive fragment of $\QGL$ and a fragment of $\QPL(\PA)$.
  \end{abstract}


\section{Introduction}

Provability is a fundamental concept in mathematics, logic and philosophy alike. Gödel proved his famous incompleteness results in \cite{Godel1931} by formalizing provability. Thus, for formal theories like Peano Arithmetic ($\PA$), we have a natural arithmetical predicate $\Box_\PA(\cdot)$ that is true exactly for the Gödel numbers of $\PA$-provable sentences:
\begin{equation}\label{equation:definitionOfProvabilityPredicate}
  \PA \vdash A \iff \mathbb N \vDash \Box_\PA (\gnum{A}).
\end{equation}
For readability we shall not distinguish formulas from their Gödel numbers or from syntactic terms (numerals) denoting these Gödel numbers in the future.

Gödel observed various provable structural properties of the provability predicate. For example, for any formula $A$, if $\PA\vdash A$ then $\PA\vdash \Box_\PA A$. Moreover, this can be formalized itself: for any formula $A$, $\PA \vdash \Box_\PA A \to \Box_\PA \Box_\PA A$. Using such properties, and after showing how self-reference can be obtained in $\PA$, Gödel derived his first incompleteness theorem (here presented in a slightly weakened form for $\PA$) by observing that the sentence $B$ such that $\PA \vdash B \leftrightarrow \neg \Box_\PA B$ can neither be proved nor refuted in $\PA$, provided that $\PA$ only proves true theorems.

In light of this, it makes sense to design a system that collects all provable structural properties of the provability predicate. The language $\mathcal L_\Box$ of propositional modal logic is optimally suited for this purpose. The formulas of this language are as the ones of propositional logic together with a unary modality $\Box$ that is syntactically treated as negation. Thus we can write ${\sf Form}_\Box ::= \bot \mid {\sf Prop} \mid {\sf Form}_\Box \to {\sf Form}_\Box \mid \Box {\sf Form}_\Box$, where $\mathsf{Prop}$ represents propositional symbols.

The modal logical formulas are linked to arithmetic using so-called \emph{realizations}. A particular realization $\cdot\fs$ maps propositional variables to sentences in the language of $\PA$ and this map is extended to any formula by stipulating that it commutes with implication, $\bot$ is mapped to $(0=1)$, and the modal operator is mapped to formalized provability, i.e.~$(\Box A)\fsPA = \Box_\PA A\fsPA$.

We can now express what it means to be a provable structural property. For example, for any realization $\cdot\fs$ we have
$\PA \vdash (\Box (p \to q) \to (\Box p \to \Box q))\fsPA$, since $\PA \vdash  \Box_\PA (p\fs \to q\fs) \to (\Box_\PA p\fs \to \Box_\PA q\fs)$ for any formulas $p\fs$ and $q\fs$. The structural properties are thus captured by the modal formulas that are provable under any realization. We call this $\PL(\PA)$, the \emph{provability logic} of $\PA$, and write
\begin{equation*}
\PL(\PA) := \{ A \in \mathcal L_\Box \mid \text{for any }\cdot\fs, \text{ we have } \PA \vdash A\fsPA \}.
\end{equation*}
Via \eqref{equation:definitionOfProvabilityPredicate} we know that ($\PA \vdash A\fsPA$ for any $\cdot\fs$) if and only if ($\N \vDash \Box_\PA A\fsPA$ for any $\cdot\fs$). Clearly $A\fsPA$ only depends on the value of $\cdot\fs$ for the finitely many propositional variables that occur in $A$, and so the universal quantifier ``for any $\cdot\fs$'' can be coded and made internal, making it possible to characterize $\PL(\PA)$ by $\{ A \mid \mathbb N \vDash \Forall{\cdot\fs} \Box_\PA A\fsPA\}$.
We now observe that $\Box_\PA B$ is a $\Sigma^0_1$ formula. That is, $\Box_\PA B$ is of the form $\Exists{p} \Prf_\PA(p, B)$, where $\Prf_\PA(\cdot, \cdot)$ is a decidable predicate. Thus, $\PL(\PA)= \{ A \mid \mathbb N \vDash \Forall{(\cdot)\fs} \Box_\PA A\fsPA\}$ has a $\Pi^0_2$ definition. If we moreover realize that $\Box_\PA (\cdot)$ is $\Sigma^0_1$ complete in that any computably enumerable\footnote{Also known as c.e.,~recursively enumerable, or computably axiomatizable.} set $U$ can be defined using a formula $A_U$ as $\{ n \mid \mathbb N \vDash \Box_\PA (A_U(n))\}$, then it seems that there is little hope that the $\Pi^0_2$ defined set $\PL(\PA)$ allows for a simple characterization. 

However, a little miracle happens and by Solovay's completeness theorem \cite{Solovay:1976} we know that $\PL(\PA)$ is decidable, corresponding to what we nowadays call the \emph{Gödel-Löb provability logic} $\gl$. Likewise, Solovay proved that the set of \emph{true} structural provability principles 
\begin{equation*}
\TPL(\PA) := \{ \varphi \in \mathcal L_\Box \mid \text{for any }\cdot\fs, \text{ we have } \mathbb N \vDash \varphi\fsPA \}
\end{equation*}
is also decidable and described by the well-behaved modal logic $\mathsf{GLS}$. We observe that, \emph{a priori}, $\TPL(\PA)$ falls outside the arithmetical hierarchy due to Tarski's \cite{Tarski1936} result on the undefinability of arithmetical truth.

After these positive results, there was hope that the nice characterizations could be extended to the realm of quantified modal logic. 
Thus, the focus switched to the language $\Lbf$ of relational quantified modal logic without identity, which contains
$\bot$, relation symbols, Boolean connectives, $\forall x$, and $\Box$, as well as the usual abbreviations such as $\Diamond$ and $\exists x$. We define arithmetical realizations $\cdot\fs$ for $\Lbf$ formulas as before with the only difference that we now map $n$-ary relation symbols to arithmetical formulas with $n$ free variables and set $(\Forall{x} \varphi)\fsPA:= \Forall{y} \varphi\fsPA$, where $y$ is the arithmetical variable corresponding to $x$.
The \emph{quantified provability logic} of $\PA$ can now be defined by analogy (cf.~Boolos \cite{Boolos:1993:LogicOfProvability}):%
  \footnote{In general, $\QPL(T)$ may change depending on the chosen axiomatization $\tau$ for $T$ (see \cite{Artemov1986, Kurahashi2013, Kurahashi2021}), but here we omit the axiomatization for simplicity.}
\begin{equation*}
  \QPL(\PA) := \{
    A \in \Lbf
    \mid
    \text{for any }\cdot\fs, \text{ we have } \PA \vdash A\fsPA
  \}
  ,
\end{equation*}
as well as 
\begin{equation*}
  \TQPL(\PA) := \{
    A \in \Lbf
    \mid
    \text{for any }\cdot\fs, \text{ we have } \mathbb N \vDash A\fsPA
  \}
  .
\end{equation*}

It is not hard to see that $\Box \Forall{x} A \to \Forall{x} \Box A$ is in $\QPL(\PA)$, and for some time it was believed that one could obtain $\QPL(\PA)$ by adding such principles together with predicate logical reasoning to $\gl$. A first wrinkle in this hope appeared when Montagna \cite{Montagna1984} published a new always provable principle falling outside this easily generated set. Soon after, Artemov \cite{Artemov1985} proved that $\TQPL(\PA)$ is not even arithmetically definable. Last hopes were scattered when Vardanyan \cite{Vardanyan1986} and McGee \cite{McGee1985} independently proved that $\QPL(\PA)$ is as complex as it can possibly be: $\Pi^0_2$ complete. Moreover, Boolos and McGee \cite{BoolosMcGee1987} showed that $\TQPL(\PA)$ is also as complex as possible: $\Pi^0_1$ complete in the set of true arithmetic.

These negative results are quite robust in various ways.
  For example, Vardanyan \cite{Vardanyan1988} showed that restricting the language to the fragment with a single unary predicate and without any modality nesting doesn't break the $\Pi^0_2$ completeness.

 Restricting the complexity of the realizations to, for example, $\Sigma^0_1$ formulas doesn't help either: the corresponding set is still $\Pi^0_2$ complete, as shown by Berarducci \cite{Berarducci1989}.

  One can consider $\QPL(T)$ for an arbitrary c.e.~theory $T$.
  It is easy to see that the degenerate case of $\QPL(\PA + \Box_\PA \bot)$ is axiomatized by predicate logic together with $\Box\bot$.%
  \footnote{Given a quantified modal formula $\varphi$, let $\widetilde{\varphi}$ be $\varphi$ with every subformula of the form $\Box \psi$ replaced by $\top$. Let $\cdot\fs$ be any realization. Then $\PA + \Box_\PA \bot \vdash \varphi\fs$ if and only if $\PA \vdash \widetilde{\varphi}\fs$ (by induction on $\varphi$). Similarly, $\mathsf{FOL} + \Box \bot \vdash \varphi$ if and only if $\mathsf{FOL} \vdash \widetilde{\varphi}$.
}
  However, Visser and de Jonge \cite{VisserAndDeJonge:2006:NoEscape} proved that for most theories $T$, $\QPL(T)$ is indeed $\Pi^0_2$ complete.
  The title of their paper is \emph{No Escape from Vardanyan's Theorem}.

  Be that as it may, there have already been some ``escapes" from Vardanyan's Theorem.
  For example, the one-variable fragment of $\QPL(\PA)$ is decidable, as shown by Artemov and Japaridze \cite{ArtemovJaparidze1990}.
  Furthermore, an arithmetically complete quantified modal logic was proposed by Yavorsky \cite{Yavorsky2002} (see also \cite{HaoTourlakis2021}). This logic, called $\QGL^b$, assumes that in $\Box A$ every free variable of $A$ is bound under the box. In other words, $\QGL^b$ is roughly $\QGL$ extended with the axiom schema $\Box A \to \Box \Forall{x} A$. 
  
  Our proposed fragment includes countably many variables and allows for open variables under the box. In order to ``escape", we restrict the language to its strictly positive fragment instead. This work already began in \cite{QRC1}, where we described the system $\QRC$ and proved its decidability. Here we improve on the modal results presented there and prove the arithmetical completeness theorem for $\QRC$, previously left as a conjecture.

  There is an ongoing formalization\footnote{\url{https://gitlab.com/ana-borges/QRC1-Coq}} of this paper in the Coq Proof Assistant \cite{coq}, covering Sections~\ref{sec:QRC} and \ref{sec:Kripke_semantics} (and part of Section~\ref{sec:constant_domain}) as of August 2021.

\subsection{Overview of the paper}

We start with a brief overview of strictly positive logics in Section~\ref{sec:strictly_positive}, followed by the definition of $\QRC$, our main object of study, in Section~\ref{sec:QRC}.

After that the paper is divided into two parts, the first dealing with purely modal results (Sections~\ref{sec:Kripke_semantics} to \ref{sec:QGL}), and the second delving into arithmetic (Sections~\ref{sec:arithmetics} to \ref{sec:HA}). The two parts are not completely independent, but a reader who wishes to skip to Section~\ref{sec:arithmetics} need only be familiar with the definition of Kripke model with constant domain (Definition~\ref{def:Kripke_model}) and the constant domain completeness theorem for $\QRC$ (Theorem~\ref{thm:constant_domain}).

In the first part, Section~\ref{sec:Kripke_semantics} starts by presenting an extended definition of Kripke model that does not depend on inclusive models, with a respective soundness proof for $\QRC$.
Then in Section~\ref{sec:constant_domain} we show that $\QRC$ is complete with respect to constant domain Kripke models.
Section~\ref{sec:QGL} remarks that $\QRC$ is the strictly positive fragment of every logic between $\QKf$ and $\QGL$, and also of other quantified modal logics such as $\QKf + \mathsf{BF}$, where $\mathsf{BF} = \Forall{x} \Box A \to \Box \Forall{x} A$ is the well-known and arithmetically unsound Barcan Formula.

In the second part, Section~\ref{sec:arithmetics} explains our arithmetic reading of strictly positive formulas, followed by Section~\ref{sec:arithmetical_completeness}, where we prove the arithmetical completeness of $\QRC$.
Section~\ref{sec:QPL} then makes use of the arithmetical completeness theorem to show that $\QRC$ is a fragment of $\QPL(\PA)$.
Section~\ref{sec:HA}, which can be read right after Section~\ref{sec:arithmetics}, shows the arithmetical soundness of $\QRC$ with respect to Heyting Arithmetic.

We end with some proposed avenues for future work in Section~\ref{sec:future}.

\section{Strictly positive logics}
\label{sec:strictly_positive}

A key feature of our escape is given by restricting the language. Given variables $x, x_i, \hdots$ and a signature $\Sigma$ fixing the constants $c, c_i, \hdots$ and relation symbols $S, S_i, \hdots$, the formulas of $\QRC$ are built up from $\top$, $n$-ary relation symbols applied to $n$ terms (which are either variables or constants), the binary $\land$, the unary $\Diamond$ and the quantifier $\forall x$. The provable judgments in $\QRC$ are all of the form $\varphi \vdash \psi$ with $\varphi$ and $\psi$ in the above language.

Our formulas are related to arithmetic through realizations $\cdot\is$ that map $n$-ary predicate symbols to c.e.~axiomatizations of theories indexed by $n$ parameters (details can be found in Section~\ref{sec:arithmetics}). Given a $\Sigma^0_1$ axiomatization $\tau$ of some theory $T$, we extend $\cdot\is$ to non-predicate formulas such that conjunctions are interpreted as the unions of the corresponding axiom sets ($(\varphi \land \psi)\isT := \varphi\isT \lor \psi\isT$), and universal quantification corresponds to an infinite union ($(\Forall{x} \varphi)\isT := \Exists{y} \varphi\isT$, where $y$ is the arithmetical variable corresponding to the modal variable $x$).

This interpretation is common in the study of reflection calculi \cite{Beklemishev2014} and is more general than the usual, finitary one, mentioned in the previous section as $\cdot\fs$. In the latter case, conjunctions of modal formulas are simply interpreted as conjunctions of their arithmetical counterparts. We further define and make use of the finitary notion in Section~\ref{sec:arithmetical_completeness}.
While it is possible to represent finite extensions of a base theory through the  finitary $\cdot\fs$-style realizations, the $\cdot\is$ approach allows for the possibility of infinitary extensions.

Using this restricted fragment and infinitary arithmetical interpretation, we define the \emph{Strictly Positive Quantified Provability Logic} of a theory $T$ as follows. 
\begin{equation*}
  \QPL^{\sf SP}(T) := \{ \la \varphi , \psi \ra \mid \Forall{\cdot\is} T \vdash \Forall{\theta} (\Box_{\psi\isT} \theta \to \Box_{\varphi\isT}\theta)\}.
\end{equation*}
The main result of this paper is that, for a large class of theories $T$, the logic $\QPL^{\sf SP}(T)$ is decidable and given by the system $\QRC$ as introduced in \cite{QRC1}. As such our paper was inspired by and contributes to three recent developments in the literature: strictly positive logics, reflection calculi, and polymodal logics.

The quintessential polymodal provability logic is $\GLP$, introduced by Japaridze in 1986 \cite{Japaridze1986}.
The Reflection Calculus, $\RC$, was first introduced by Dashkov \cite{Dashkov2012} as the set of $\GLP$-equivalences between strictly positive formulas. It was then axiomatized by Beklemishev \cite{Beklemishev2012}, and it is the latter formulation that appears in most of the ensuing literature.

  Even though $\RC$ has a strictly positive language, it is remarkably expressive, giving rise to an ordinal notation system \cite{FernandezDuque2017} and being an appropriate tool for $\Pi^0_1$ ordinal analysis \cite{BeklemishevPakhomov:2019:GLPforTheoriesOfTruth}.
  This is remarkable because, while $\GLP$ is $\mathsf{PSPACE}$ complete \cite{Shapirovsky2008}, $\RC$ has a polynomial-time decision procedure.
  Thus, at least in the case of $\GLP$, restricting the language to the strictly positive fragment is very worthwhile.

  $\QRC$ was inspired by $\RC_1$, which is the unimodal fragment of $\RC$. We hoped to emulate $\RC$'s success in these two dimensions, obtaining a useful calculus with a simpler complexity than the original ($\QPL(\PA)$ in this case). We already see in this paper that the latter goal was achieved, since $\QRC$ is decidable while $\QPL(\PA)$ is $\Pi^0_2$ complete.
  However, the main expressibility tool available in $\RC$, the iterated consistency statements (also known as worms \cite{Beklemishev2006_WormPrinciple}), are not very interesting when there is only one modality. Thus we plan to extend $\QRC$ to $\mathsf{QRC}_\Lambda$ in the future.

  When considering a strictly positive logic $P$, one may ask whether there is some modal logic $L$ whose strictly positive fragment $L^{\mathsf{SP}}$ coincides with $P$ (cf.~\cite{Bek18b}). In that case we would have
\begin{equation}
\label{eq:Lsp}
  \varphi \vdash_P \psi
  \iff
  \varphi \vdash_{L^{\mathsf{SP}}} \psi
  \iff
  L \vdash \varphi \to \psi
  ,
\end{equation}
where $\varphi$ and $\psi$ are strictly positive formulas.

We know that $\RC$ is the strictly positive fragment of $\GLP$ (in the sense of \eqref{eq:Lsp}), and in Section~\ref{sec:QGL} we show that $\QRC$ is the strictly positive fragment of $\QGL$.
However, there are strictly positive logics with no such counterpart, such as $\RC$ together with the persistence axiom $\la \omega \ra \varphi \vdash \varphi$. This system was described in \cite{Beklemishev2014} as $\RC\omega$.

  The study of strictly positive logics has also been developed in other communities, most notably in the field of description logics, as in \cite{Kurucz2010, Kikot2019}.
  Furthermore, there is a significant body of work done about (non strictly) positive logics (see \cite{Dunn1995, CelaniJansana2012}).

\section{Quantified Reflection Calculus with one modality}
\label{sec:QRC}

The \emph{Quantified Reflection Calculus with one modality}, or $\QRC$, is a sequent logic in a strictly positive predicate modal language introduced in \cite{QRC1}.

The free variables of a formula $\varphi$ are defined as usual, and denoted by $\fv(\varphi)$. The expression $\varphi\subst{x}{t}$ denotes the formula $\varphi$ with all free occurrences of the variable $x$ simultaneously replaced by the term $t$. We say that $t$ is free for $x$ in $\varphi$ if no occurrence of a free variable in $t$ becomes bound in $\varphi\subst{x}{t}$.

The axioms and rules of $\QRC$ are listed in the following definition from \cite{QRC1}. Here we removed the axiom $\Diamond \Forall{x} \varphi \vdash \Forall{x} \Diamond \varphi$ because it is an easy consequence of the calculus without it.

\begin{definition}[$\QRC$, \cite{QRC1}]
  Let $\Sigma$ be a signature and $\varphi$, $\psi$, and $\chi$ be any formulas in that language. The axioms and rules of $\QRC$ are the following:

\begin{multicols}{2}
\begin{enumerate}[label=\upshape(\roman*),ref=\thetheorem.(\roman*)]
  \item $\varphi \vdash \top$ and $\varphi \vdash \varphi$;
  \item $\varphi \land \psi \vdash \varphi$ and $\varphi \land \psi \vdash \psi$;
  \item if $\varphi \vdash \psi$ and $\varphi \vdash \chi$, then\\$\varphi \vdash \psi \land \chi$;
  \item if $\varphi \vdash \psi$ and $\psi \vdash \chi$, then $\varphi \vdash \chi$;\label{rule:cut}
  \item if $\varphi \vdash \psi$, then $\Diamond \varphi \vdash \Diamond \psi$;\label{rule:necessitation}
    \item $\Diamond \Diamond \varphi \vdash \Diamond \varphi$; \label{ax:transitivity}
  \item if $\varphi \vdash \psi$, then $\varphi \vdash \Forall{x} \psi$\\($x \notin \fv(\varphi)$); \label{rule:forallR}
  \item if $\varphi\subst{x}{t} \vdash \psi$ then $\Forall{x} \varphi \vdash \psi$\\($t$ free for $x$ in $\varphi$);\label{rule:forallL}
  \item if $\varphi \vdash \psi$, then $\varphi\subst{x}{t} \vdash \psi\subst{x}{t}$\\($t$ free for $x$ in $\varphi$ and $\psi$);\label{rule:term_instantiation}
  \item if $\varphi\subst{x}{c} \vdash \psi\subst{x}{c}$, then $\varphi \vdash \psi$\\($c$ not in $\varphi$ nor $\psi$).\label{rule:constants}
\end{enumerate}
\end{multicols}

If $\varphi \vdash \psi$, we say that $\psi$ follows from $\varphi$ in $\QRC$. When the signature is not clear from the context, we write $\varphi \vdash_\Sigma \psi$ instead.
\end{definition}

We observe that our axioms do not include universal quantifier elimination. However, this and various other rules are readily available via the following easy lemma.
\begin{lemma}
\label{lem:QRC1consequences}
  The following are theorems (or derivable rules) of $\QRC$:
  \begin{enumerate}[label=\upshape(\roman*),ref=\thetheorem.(\roman*)]
    \item $\Forall{x} \Forall{y} \varphi \vdash \Forall{y} \Forall{x} \varphi$;
    \item $\Forall{x} \varphi \vdash \varphi\subst{x}{t}$ ($t$ free for $x$ in $\varphi$); \label{lem:instantiation}
    \item $\Diamond \Forall{x} \varphi \vdash \Forall{x} \Diamond \varphi$;\footnote{In \cite{QRC1} this was presented as an axiom of $\QRC$, but it is easily provable through the two $\forall$ introduction rules and necessitation.}
    \item $\Forall{x} \varphi \vdash \Forall{y} \varphi\subst{x}{y}$ ($y$ free for $x$ in $\varphi$ and $y \notin \fv(\varphi)$);
    \item if $\varphi \vdash \psi$, then $\varphi \vdash \psi\subst{x}{t}$ ($x$ not free in $\varphi$ and $t$ free for $x$ in $\psi$);
    \item if $\varphi \vdash \psi\subst{x}{c}$, then $\varphi \vdash \Forall{x} \psi$ ($x$ not free in $\varphi$ and $c$ not in $\varphi$ nor $\psi$). \label{item:constants_forall}
  \end{enumerate}
\end{lemma}

The following are two useful complexity measures on the formulas of $\QRC$.

\begin{definition}[$\mdepth$, $\udepth$, \cite{QRC1}]
  Given a formula $\varphi$, its \emph{modal depth} $\mdepth(\varphi)$ is defined inductively as follows:
  \begin{itemize}
    \item $\mdepth(\top) := \mdepth(S(x_0, \hdots, x_{n-1})) := 0$;
    \item $\mdepth(\psi \land \chi) := \max\{\mdepth(\psi), \mdepth(\chi)\}$;
    \item $\mdepth(\Forall{x} \psi) := \mdepth(\psi)$;
    \item $\mdepth(\Diamond \psi) := \mdepth(\psi) + 1$.
  \end{itemize}
  Given a finite set of formulas $\Gamma$, its modal depth is $\mdepth(\Gamma) := \max_{\varphi \in \Gamma}\{\mdepth(\varphi)\}$.

  The definition of quantifier depth $\udepth$ is analogous except for:
  \begin{itemize}
    \item $\udepth(\Forall{x} \psi) = \udepth(\psi) + 1$; and
    \item $\udepth(\Diamond \psi) = \udepth(\psi)$.
  \end{itemize}
\end{definition}

The modal depth provides a necessary condition for derivability, which in particular implies irreflexivity.

\begin{lemma}[\cite{QRC1}]
\label{lem:mdepth}

Let $\varphi$ and $\psi$ be formulas in the language of $\QRC$.
  \begin{itemize}
    \item If $\varphi \vdash \psi$, then $\mdepth(\varphi) \geq \mdepth(\psi)$.
    \item $\varphi \not\vdash \Diamond \varphi$.
  \end{itemize}
\end{lemma}

Finally, the signature of $\QRC$ can be extended without strengthening the calculus.
\begin{lemma}[\cite{QRC1}]
\label{lem:signatures}
  Let $\Sigma$ be a signature and let $C$ be a collection of constants not yet occurring in $\Sigma$. By $\Sigma_C$ we denote the signature obtained by including these new constants $C$ in $\Sigma$. Let $\varphi, \psi$ be formulas in the language of $\Sigma$. Then, if $\varphi \vdash_{\Sigma_C} \psi$, so does $\varphi \vdash_\Sigma \psi$.
\end{lemma}

\section{Relational semantics}
\label{sec:Kripke_semantics}

$\QRC$ was proven sound and complete with respect to relational semantics in \cite{QRC1}. Here we extend both of those results in the following ways: we relax the requirement for the adequateness of a frame, and we prove constant domain completeness: that if $\varphi \not\vdash \psi$ then there exists a counter model that, in addition to being finite and irreflexive, also has a constant domain.

We begin by slightly changing the definition of frame and relational model presented in \cite{QRC1}. There, models for $\QRC$ were described as a number of first-order models (the worlds) connected through a transitive relation $R$. We additionally required \emph{inclusiveness}: that whenever $w$ and $u$ are worlds connected through $R$, the domain of $w$ be included in the domain of $u$. We then used the inclusion (identity) function $\iota_{w, u}$ to refer to the element of the domain of $u$ corresponding to an element in the domain of $w$.

Here, we no longer have the inclusiveness restriction. In fact, as we will see bellow, any configuration of domains is sound as long as the functions $\eta_{w, u}$ relating the domain of $w$ with the domain of $u$ respect the transitivity of $R$. This is clearly the case if the frame is inclusive and $\eta_{w, u} = \iota_{w, u}$, so the definitions presented in \cite{QRC1} are a particular case of the ones presented here.

\begin{definition}
\label{def:Kripke_model}

  A \emph{relational model} $\M$ in a signature $\Sigma$ is a tuple $\la W, R,\allowbreak \{M_w\}_{w \in W},\allowbreak \{\eta_{w, v}\}_{wRv}, \{I_w\}_{w \in W},\allowbreak \{J_w\}_{w \in W} \ra$ where:
  \begin{itemize}
    \item $W$ is a non-empty set (the set of worlds, where individual worlds are referred to as $w, u, v$, etc);
    \item $R$ is a binary relation on $W$ (the accessibility relation);
    \item each $M_w$ is a finite set (the domain of the world $w$, whose elements are referred to as $d, d_0, d_1$, etc);
    \item if $wRv$, then $\eta_{w, v}$ is a function from $M_w$ to $M_v$;
    \item for each $w \in W$, the interpretation $I_w$ assigns an element of the domain $M_w$ to each constant $c \in \Sigma$, written $c^{I_w}$; and
    \item for each $w \in W$, the interpretation $J_w$ assigns a set of tuples $S^{J_w} \subseteq \wp((M_w)^n)$ to each $n$-ary relation symbol $S \in \Sigma$.
  \end{itemize}

  The $\la W, R, \{M_w\}_{w \in W}, \{\eta_{w, v}\}_{wRv} \ra$ part of the model is called its \emph{frame}. We say that the frame (or model) is \emph{finite} if $W$ is finite, and that it is \emph{constant domain} if all the $M_w$ coincide and all the $\eta_{w, u}$ are the identity function.
\end{definition}

The relevant frames and models will need to satisfy a number of requisites.

\begin{definition}
  A frame $\F$ is \emph{adequate} if:
  \begin{itemize}
    \item $R$ is transitive: if $wRu$ and $uRv$, then $wRv$; and
    \item the $\eta$ functions respect transitivity: if $wRu$ and $uRv$, then $\eta_{w, v} = \eta_{u, v} \circ \eta_{w, u}$\footnote{We only require extensional equality.}.
  \end{itemize}
  A model is \emph{adequate} if it is based on an adequate frame and it is:
  \begin{itemize}
    \item concordant: if $wRu$, then $c^{I_u} = \eta_{w, u}(c^{I_w})$ for every constant $c$.
  \end{itemize}
  Note that in an adequate and rooted model the interpretation of the constants is fully determined by their interpretation at the root.
\end{definition}

As in \cite{QRC1}, we use assignments to define truth at a world in a first-order model. A $w$-assignment $g$ is a function assigning a member of the domain $M_w$ to each variable in the language. 
Any $w$-assignment can be seen as a $v$-assignment as long as $wRv$, by composing it with $\eta_{w, v}$ on the left. We write $g^v$ to shorten $\eta_{w, v} \circ g$ when $w$ is clear from the context.

Two $w$-assignments $g$ and $h$ are \emph{$x$-alternative}, written $g \xaltern{x} h$, if they coincide on all variables other than $x$.
A $w$-assignment $g$ is extended to terms by defining $g(c) := c^{I_w}$ for any constant $c$. Note that this meshes nicely with the concordant restriction of an adequate model: for any term $t$, if $wRv$ then $g^v(t) = \eta_{w, v}(g(t))$.

We are finally ready to define satisfaction at a world. This definition is a straightforward adaptation of the one presented in \cite{QRC1} to our current definition of model. The only difference is in the case of $\Diamond \varphi$, where we use $g^v$ instead of $g^\iota$.

\begin{definition}
  Let $\M = \la W, R, \{M_w\}_{w \in W}, \{\eta_{w, u}\}_{wRu}, \{I_w\}_{w \in W}, \{J_w\}_{w \in W} \ra$ be an adequate model in some signature $\Sigma$, and let $w \in W$ be a world, $g$ be a $w$-assignment, $S$ be an $n$-ary relation symbol, and $\varphi, \psi$ be formulas in the language of $\Sigma$.

  We define $\M, w \Vdash^g \varphi$ ($\varphi$ is true at $w$ under $g$) by induction on $\varphi$ as follows.
  \begin{itemize}

    \item $\M, w \Vdash^g \top$;

    \item $\M, w \Vdash^g S(t_0, \hdots, t_{n-1})$ iff $\la g(t_0), \hdots, g(t_{n-1}) \ra \in S^{J_w}$;

    \item $\M, w \Vdash^g \varphi \land \psi$ iff both $\M, w \Vdash^g \varphi$ and $\M, w \Vdash^g \psi$;

    \item $\M, w \Vdash^g \Diamond \varphi$ iff there is a $v \in W$ such that $wRv$ and $\M, v \Vdash^{g^v} \varphi$;

    \item $\M, w \Vdash^g \Forall{x} \varphi$ iff for all $w$-assignments $h$ such that $h \xaltern{x} g$, we have $\M, w \Vdash^{h} \varphi$.

  \end{itemize}
\end{definition}

\begin{theorem}[Relational soundness]
\label{thm:modal_soundness}
If $\varphi \vdash \psi$, then for any adequate model $\M$, for any world $w \in W$, and for any $w$-assignment $g$:
  \begin{gather*}
    \M, w \Vdash^g \varphi
    \implies
    \M, w \Vdash^g \psi
    .
  \end{gather*}
\end{theorem}
\begin{proof}
  By induction on the proof of $\varphi \vdash \psi$, making the same arguments as in \cite{QRC1}. Here we highlight only the transitivity axiom, where the transitivity of the $\eta$ functions comes into play, and also remark on the generalization on constants rule.

  The transitivity axiom is $\Diamond \Diamond \varphi \vdash \Diamond \varphi$, so assume that $\M, w \Vdash^g \Diamond \Diamond \varphi$. Then there is a world $v$ such that $wRv$ and $\M, v \Vdash^{\eta_{w, v} \circ g} \Diamond \varphi$, and also a subsequent world $u$ such that $vRu$ and $\M, u \Vdash^{\eta_{v, u} \circ (\eta_{w, v} \circ g)} \varphi$. Since $R$ is transitive, we know that $wRu$ and thus that $\eta_{v, u} \circ (\eta_{w, v} \circ g)$ coincides with $\eta_{w, u} \circ g$. Then $\M, u \Vdash^{\eta_{w, u} \circ g} \psi$, and consequently $\M, w \Vdash^{g} \Diamond \varphi$, as desired.

  The soundness of the generalization on constants rule, Rule~\ref{rule:constants}, is the most involved part of the proof presented in \cite{QRC1} and depends on the construction of a model where the interpretation of a constant is changed. Building that model in this context is a simple matter of taking care to propagate that change to all future worlds using the $\eta$ functions.
\end{proof}

We end this section by noting that, even though the language of $\QRC$ is quite restricted, even its $\forall$ fragment requires counter-models with arbitrarily large domains. For example, the sequent
\begin{equation*}
  \Forall{x, y} S(x, x, y)
  \land \Forall{x, y} S(x, y, x)
  \land \Forall{x, y} S(y, x, x)
  \vdash
  \Forall{x, y, z} S(x, y, z)
\end{equation*}
is unprovable in $\QRC$, but satisfied by every world with at most two domain elements. This reasoning can be extended to any $n$: if $S$ is an $n$-ary predicate symbol, let $\varphi_n$ be the conjunction of the $n(n-1)/2$ formulas of the form $\Forall{x_0, \hdots, x_{n-2}} S(\hdots, x_0, \hdots, x_0, \hdots)$, with $x_0$ appearing in every possible pair of positions and every other position filled by a unique variable. Then $\varphi_n$ does not entail $ \psi_n := \Forall{x_0, \hdots, x_{n-1}} S(x_0, \hdots, x_{n-1})$ but any world with at most $n-1$ domain elements that satisfies $\varphi_n$ must also satisfy $\psi_n$.

\section{Constant domain completeness}
\label{sec:constant_domain}

In \cite{QRC1} we proved relational completeness by building a term model that satisfies $\varphi$ and doesn't satisfy $\psi$ when $\varphi \not\vdash \psi$. That construction provides a finite, irreflexive, and rooted model with increasing domains. Here we show that it is possible to build a constant domain model instead, that is, a model where the domain of every world is exactly the same. This is extremely useful to prove the arithmetical completeness theorem in Section~\ref{sec:arithmetical_completeness}.

Before starting the formal proof, we briefly describe the main idea. The term models we build are such that each world is a pair of sets of closed formulas $p = \pair{p^+}{p^-}$. The first set, $p^+$, is the set of formulas that will be satisfied at that world, or the \emph{positive part}. The second set, $p^-$, is the set of formulas that will not be satisfied at that world, or the \emph{negative part}. All worlds must be well-formed with respect to some set of closed formulas $\Phi$, which means that:
\begin{itemize}
  \item $p$ is closed: every formula in $p^+$ and every formula in $p^-$ is closed;
  \item $p$ is $\Phi$-maximal: every formula of $\Phi$ is in either $p^+$ or $p^-$ (and there are no formulas in $p$ but not in $\Phi$);
  \item $p$ is consistent: if $\delta \in p^-$ then $\bigwedge p^+ \not\vdash \delta$; and
  \item $p$ is fully-witnessed: if $\Forall{x} \varphi \in p^-$ then there is a constant $c$ such that $\varphi\subst{x}{c} \in p^-$.
\end{itemize}
In that case we say that $p$ is $\Phi$-maximal consistent and fully witnessed, or $\Phi$-MCW for short (the closeness condition is included in the concept, although in practice almost every formula in this section will be closed). If $p$ and $q$ are pairs, we write $p \subseteq q$ when $p^+ \subseteq q^+$ and $p^- \subseteq q^-$. Furthermore, if $\Gamma$ is a set of formulas we write $p \subseteq \Gamma$ instead of $p^+ \cup p^- \subseteq \Gamma$.

We want to have $\Phi$-MCW pairs where $\Phi$ is closed under subformulas. However, the naive subformulas of $\Forall{x} \varphi$ might be open. In order to avoid that, we use the notion of closure with respect to a set of constants $C$ defined in \cite{QRC1}, where $\varphi\subst{x}{c}$ is a valid subformula of $\Forall{x} \varphi$ as long as $c \in C$.

\begin{definition}[$\Cl_C$, \cite{QRC1}]
  Given a set of constants $C$, the closure of a formula $\varphi$ under $C$, written $\Cl_C(\varphi)$, is defined by induction on the formula as such: $\Cl_C(\top) := \{\top\}$; $\Cl_C(S(t_0, \hdots, t_{n-1})) := \{S(t_0, \hdots, t_{n-1})), \top\}$; $\Cl_C(\varphi \land \psi) := \{\varphi \land \psi\} \cup \Cl_C(\varphi) \cup \Cl_C(\psi)$; $\Cl_C(\Diamond \varphi) := \{\Diamond \varphi\} \cup \Cl_C(\varphi)$; and
  \begin{equation*}
    \Cl_C(\Forall{x} \varphi) :=
      \{\Forall{x} \varphi\}
      \cup
      \bigcup_{c \in C} \Cl_C(\varphi\subst{x}{c})
    .
  \end{equation*}

  The closure under $C$ of a set of formulas $\Gamma$ is the union of the closures under $C$ of each of the formulas in $\Gamma$:
  \begin{equation*}
    \Cl_C(\Gamma) := \bigcup_{\gamma \in \Gamma} \Cl_C(\gamma).
  \end{equation*}
  The closure of a pair $p$ is defined as the closure of $p^+ \cup p^-$.
\end{definition}

Going back to the overview of the completeness proof, suppose that $\varphi \not\vdash \psi$, (assuming for now that $\varphi$ and $\psi$ are closed). Defining $p := \pair{\{\varphi\}}{\{\psi\}}$, the counter-model will be rooted on a $\Cl_C(p)$-MCW extension of $p$, where $C$ is a set of constants to determine. The set of constants $C$ will be used as the domain of the root. Note that $p$ is already closed and consistent, so taking the step to maximality is as simple as deciding whether to add each $\chi \in \Cl_C(p)$ to the positive or negative part of the root without ruining its consistency. The hard part is doing so in a way that guarantees that the resulting pair is fully witnessed.

The completeness proof shown in \cite{QRC1} uses the observation that if $\bigwedge p^+ \not\vdash \Forall{x} \chi$, then also $\bigwedge p^+ \not\vdash \chi\subst{x}{c}$, as long as $c$ does not appear in either $p^+$ or $\chi$ (Lemma~\ref{item:constants_forall}). This suggests a way of sorting the formulas of $\Cl_C(p)$ into positive and negative: mark a formula as positive if and only if it is a consequence of $p^+$. This guarantees that there are witnesses for the negative universal formulas as long as there are enough constants to go around. The way to make sure that there are enough constants is precisely the difference between the proof presented in \cite{QRC1} and the proof presented here. To that end, we introduce the following definition.

\begin{definition}[$\cdepth$]
  The number of different constants in a formula $\varphi$ is represented by $\cdepth(\varphi)$.
  The maximum number of different constants per formula in a set of formulas $\Gamma$ is defined as $\cdepth(\Gamma) := \max_{\varphi \in \Gamma}\{\cdepth(\varphi)\}$.
\end{definition}
Note that $\cdepth(\Gamma)$ is \emph{not} the number of different constants appearing in $\Gamma$, but the maximum number of different constants in any single formula of $\Gamma$. For example, $\cdepth(\{S_0(c_0, c_1), S_1(c_2)\}) = 2$.

We observe that the maximum number of distinct constants per formula in the closure under $C$ of a set of formulas can be bounded by a number that \emph{does not depend on $C$}.

\begin{remark}
\label{rem:cdepth}

For any formula $\varphi$, set of formulas $\Phi$, and set of constants $C$:
  \begin{itemize}
    \item $\cdepth(\varphi) \leq \cdepth(\varphi\subst{x}{c}) \leq \cdepth(\varphi) + 1$;
    \item $\cdepth(\varphi) \leq \cdepth(\Cl_C(\varphi)) \leq \cdepth(\varphi) + \udepth(\varphi)$;
    \item $\cdepth(\Phi) \leq \cdepth(\Cl_C(\Phi)) \leq \cdepth(\Phi) + \udepth(\Phi)$.
  \end{itemize}
\end{remark}

We are now ready to prove a Lindenbaum-like lemma.

\begin{lemma}
\label{lem:lindenbaum}
Given a finite signature $\Sigma$ with constants $C$ and a finite set of closed formulas $\Phi$ in the language of $\Sigma$ such that $|C| > 2\cdepth(\Phi) + 2\udepth(\Phi)$, if $p \subseteq \Cl_C(\Phi)$ is a closed consistent pair and $p^+$ is a singleton, then there is a pair $q \supseteq p$ in the language of $\Sigma$ such that $q$ is $\Cl_C(\Phi)$-MCW, and $\mdepth(q^+) = \mdepth(p^+)$. 
\end{lemma}
\begin{proof}
  Like in \cite{QRC1}, we start by defining a pair $q \supseteq p$ such that for each $\chi \in \Cl_C(\Phi)$, $\chi \in q^+$ if and only if $p^+ \vdash \chi$ (otherwise $\chi \in q^-$). It is easy to see that this pair $q$ is $\Cl_C(\Phi)$-maximal consistent, and we have $\mdepth(q^+) = \mdepth(p^+)$ by Lemma~\ref{lem:mdepth}.

  It remains to show that $q$ is fully witnessed. Let $\Forall{x} \psi$ be a formula in $q^-$. We claim that there is $d \in C$ such that $d$ does not appear either in $p^+$ or in $\Forall{x} \psi$. For this it is enough to see that $|C| > \cdepth(\bigwedge p^+) + \cdepth(\Forall{x} \psi)$, which is the same as $|C| > \cdepth(p^+) + \cdepth(\Forall{x} \psi)$ because $p^+$ is a singleton by assumption. Since $p^+ \subseteq \Cl_C(\Phi)$, we know that $\cdepth(p^+) \leq \cdepth(\Cl_C(\Phi))$, and similarly for $\Forall{x} \psi$. Then by Remark~\ref{rem:cdepth} we may conclude that $\cdepth(p^+) + \cdepth(\Forall{x} \psi) \leq 2\cdepth(\Phi) + 2\udepth(\Phi)$, which suffices by our assumption on the size of $C$.

  Since $d \in C$ does not appear in either $p^+$ or $\Forall{x} \psi$, we conclude that $p^+ \not\vdash \psi\subst{x}{d}$ by Lemma~\ref{item:constants_forall}, and consequently $\psi\subst{x}{d} \in q^-$ as desired.
\end{proof}

We now recall the definition of $\hat{R}$ from \cite{QRC1}, which is the relation we use to connect the worlds of the term model.

\begin{definition}[$p\hat{R}q$, \cite{QRC1}]
  The relation $\hat{R}$ between pairs is such that $p\hat{R}q$ if and only if both of following hold:
  \begin{itemize}
    \item for any formula $\Diamond \varphi \in p^-$ we have $\varphi, \Diamond \varphi \in q^-$; and
    \item there is some formula $\Diamond \psi \in p^+ \cap q^-$.
  \end{itemize}
\end{definition}

\begin{lemma}[\cite{QRC1}]
\label{lem:hatR-transitive-irreflexive}
  The relation $\hat{R}$ restricted to consistent pairs is transitive and irreflexive.
\end{lemma}

We now see that if $w$ is a $\Cl_C(\Phi)$-MCW pair with $\Diamond \varphi \in w^+$, we can find a $\Cl_C(\Phi)$-MCW pair $v$ with $\varphi \in v^+$ and $w\hat{R}v$. The proof is the same as in \cite{QRC1}, except that we now use Lemma~\ref{lem:lindenbaum} to obtain constant domains throughout the model.

\begin{lemma}[Pair existence]
\label{lem:pair_existence}
Let $\Sigma$ be a signature with a finite set of constants $C$, and $\Phi$ be a finite set of closed formulas in the language of $\Sigma$ such that $|C| > 2\cdepth(\Phi) + 2\udepth(\Phi)$. If $p$ is a $\Cl_C(\Phi)$-MCW pair and $\Diamond \varphi \in p^+$, then there is a $\Cl_{C}(\Phi)$-MCW pair $q$ such that $p\hat{R}q$, $\varphi \in q^+$, and $\mdepth(q^+) < \mdepth(p^+)$.
\end{lemma}
\begin{proof}
  Consider the pair $r$ defined as $r^+ := \{\varphi\}$ and $r^- := \{\delta, \Diamond \delta \mid \Diamond \delta \in p^-\} \cup \{\Diamond \varphi\}$. It is easy to check that $r$ is consistent and that $p\hat{R}r$ (details can be found in \cite{QRC1}), and clearly $r \subseteq \Cl_C(\Phi)$.
  We then use Lemma~\ref{lem:lindenbaum} to obtain a $\Cl_{C}(\Phi)$-MCW pair $q \supseteq r$ such that $\mdepth(q^+) = \mdepth(r^+) = \mdepth(\varphi) < \mdepth(p^+)$. We obtain $p\hat{R}q$ as a straightforward consequence of $p\hat{R}r$.
\end{proof}

We can now define an adequate and constant domain model $\M[p]$ from any given finite and consistent pair $p$ such that $\M[p]$ satisfies the formulas in $p^+$ and doesn't satisfy the formulas in $p^-$.
The idea is exactly the same as in \cite{QRC1}: build a term model where each world $w$ is a $\Cl_{M}(p)$-MCW pair, and the worlds are related by (a sub-relation of) $\hat{R}$.

\begin{definition}
\label{def:Mp}
Let $\Sigma$ be a signature. Given a finite consistent pair $p$ of closed formulas in $\Sigma$ such that $p^+$ is a singleton\footnote{This is without loss of generality, as otherwise we could take the conjunction of every formula in $p^+$ as the new $p^+$.}, we define an adequate model $\M[p]$. Here we will use $\Phi := p^+ \cup p^-$.

Let $C$ be a set of at least $2\cdepth(\Phi) + 2\udepth(\Phi) + 1$ different constants, including all of the ones appearing in $\Sigma$ and adding more if necessary. The pairs of formulas we work with are in the signature $\Sigma$ extended by $C$.
  
We start by defining the underlying frame in an iterative manner. The root is given by Lemma~\ref{lem:lindenbaum} applied to $C$ and $p$, obtaining the $\Cl_C(\Phi)$-MCW pair $q$. Frame $\F^0$ is then defined such that its set of worlds is $W^0 := \{q\}$, its relation $R^0$ is empty, and the domain of $q$ is $M^0_{q} := C$.

Assume now that we already have a frame $\F^i$, and we set out to define $\F^{i+1}$ as an extension of $\F^i$. For each leaf $w$ of $\F^i$, i.e., each world such that there is no world $v \in \F^i$ with $w R^i v$, and for each formula $\Diamond \varphi \in w^+$, use Lemma~\ref{lem:pair_existence} to obtain a $\Cl_C(\Phi)$-MCW pair $v$ such that $w\hat{R}v$, $\varphi \in v^+$, and $\mdepth(v^+) < \mdepth(w^+)$. Now add $v$ to $W^{i+1}$, add $\la w, v \ra$ to $R^{i+1}$, define $M^{i+1}_v$ as $C$, and define $\eta_{w, v}$ as the identity function.

The process described above terminates because each pair is finite and the modal depth of $\Cl_C(\Phi)$ (and consequently of $w^+$, for any $w \subseteq \Cl_C(\Phi)$) is also finite. Thus there is a final frame $\F^m$, for some natural number $m$. This frame is constant domain by construction, but not transitive. We obtain $\F[p]$ as the transitive closure of $\F^m$, which is clearly still constant domain. The $\eta$ functions are all the identity in $C$, thus satisfying the transitivity condition. We conclude that the frame $\F[p]$ is adequate.

  In order to obtain the model $\M[p]$ based on the frame $\F[p]$, let $I_q$ take constants in $\Sigma$ to their corresponding version as domain elements. If $w$ is any other world, let $I_w$ coincide with $I_q$. This is necessary to make sure that the model is concordant, because $q$ sees every other world, and is sufficient to see that $\M[p]$ is adequate.
  Finally, given an $n$-ary predicate letter $S$ and a world $w$, define $S^{J_w}$ as the set of $n$-tuples $\la d_0, \hdots, d_{n-1}\ra \subseteq (M_w)^n$ such that $S(d_0, \hdots, d_{n-1}) \in w^+$.
\end{definition}

Since everything up until now was meant for closed formulas, and furthermore we are potentially adding new constants to the signature of the formulas we care about, we provide a way of replacing the free variables of a formula with constants.

\begin{definition}[$\varphi^g$, \cite{QRC1}]
  Given a formula $\varphi$ in a signature $\Sigma$ and a function $g$ from the set of variables to a set of constants in some signature $\Sigma' \supseteq \Sigma$, we define the formula $\varphi^g$ in the signature $\Sigma'$ as $\varphi$ with each free variable $x$ simultaneously replaced by $g(x)$.
\end{definition}

The constant domain model defined above coincides with the non-constant domain definition provided in \cite{QRC1} in everything other than that it refers to the stronger Lemmas~\ref{lem:lindenbaum} and~\ref{lem:pair_existence} that keep the domain constant. Thus the Truth Lemma holds with exactly the same proof.

\begin{lemma}[Truth lemma, \cite{QRC1}]
\label{lem:modal_truth}
Let $\Sigma$ be a signature. For any finite non-empty consistent pair $p$ of closed formulas in the language of $\Sigma$, world $w \in \M[p]$, $w$-assignment $g$, and formula $\varphi$ in the language of $\Sigma$ such that $\varphi^g \in \Cl_{M_w}(p)$, we have that
  \begin{equation*}
    \M[p], w \Vdash^g \varphi
      \iff
      \varphi^g \in w^+
    .
  \end{equation*}
\end{lemma}

We are now ready to prove the constant domain completeness theorem.

\begin{theorem}[Constant domain completeness]
\label{thm:constant_domain}
  Let $\Sigma$ be a signature and $\varphi, \psi$ formulas in $\Sigma$. If $\varphi \not\vdash \psi$, then there is an adequate, finite, irreflexive, and constant domain model $\M$, a world $w \in W$, and a $w$-assignment $g$ such that:
  \begin{equation*}
    \M, w \Vdash^g \varphi
    \quad \text{and} \quad
    \M, w \not\Vdash^g \psi
    .
  \end{equation*}
\end{theorem}
\begin{proof}
  As in the original proof of relational completeness presented in \cite{QRC1}, define a new constant $c_x$ for each free variable $x$ of $\varphi$ and $\psi$ and let $\Sigma'$ be the signature $\Sigma$ augmented with these new constants and a dummy constant $c_0$.
  Let $g$ be the assignment taking each free variable $x$ of $\varphi$ and $\psi$ to $c_x$ and every other variable to $c_0$.
  Note that $\varphi^g \not\vdash_{\Sigma'} \psi^g$ by Rule~\ref{rule:constants} and Lemma~\ref{lem:signatures}.
  Build $\M := \M[\pair{\{\varphi^g\}}{\{\psi^g\}}]$ as described in Definition~\ref{def:Mp}, with root $w$. Then by Lemma~\ref{lem:modal_truth} we have both $\M, w \Vdash^g \varphi$ and $\M, w \not\Vdash^g \psi$, as desired.
\end{proof}

\section{The strictly positive fragment of \texorpdfstring{$\QKf$}{QK4} and \texorpdfstring{$\QGL$}{QGL}}
\label{sec:QGL}

Consider the language of full quantified modal logic with no equality nor constant nor function symbols, $\Lbf$. We use upper case letters to refer to formulas in this language. A propositional modal logic $\S$ can be extend to a quantified modal logic $\QS$ as described in \cite{HughesCresswell1996} by extending the language and adding the following axiom schema and rule:
\begin{itemize}
  \myitem[$\forall \mathsf{E}$] $[\forall \mathsf{E}]$
    $\Forall{x} A \to A\subst{x}{y}$ ($y$ free for $x$ in $A$);\label{ax:QforallE}
  \myitem[$\forall \mathsf{I}$] $[\forall \mathsf{I}]$
    if $\QS \vdash A \to B$, then $\QS \vdash A \to \Forall{x} B$ ($x \notin \fv(A)$).\label{rule:QforallI}
\end{itemize}

The logic $\Kf$ is obtained from the smallest normal modal logic $\mathsf{K}$ by adding the 4 axiom $\Box A \to \Box \Box A$. The logic $\gl$ is $\Kf$ extended with Löb's axiom $\Box(\Box A \to A) \to \Box A$. Dashkov \cite{Dashkov2012} showed that $\RC_1$ is the strictly positive fragment of both $\Kf$ and $\gl$ (and of any logic between them), in the sense that, if $\varphi$ and $\psi$ are built up from $\top$, propositional symbols, conjunctions, and diamonds, then $\varphi \vdash_{\RC_1} \psi$ if and only if $\Kf \vdash \varphi \to \psi$, and similarly for $\gl$. In this section we show that this equivalence is maintained in the predicate case.

Since $\QRC$ includes constants in its language and predicate modal logics are often presented without them, we define a map from the language of $\QRC$ to its constant-free subset, replacing constants with fresh variables. Thus, given a finite set of $\QRC$ formulas $\Phi$ where the constants appearing in $\Phi$ are $\vec{c} = c_0, \hdots, c_{n-1}$, let $\vec{x} = x_0, \hdots, x_{n-1}$ be fresh variables with respect to $\Phi$. Then we define $\varphi^\Phi := \varphi\subst{c_0}{x_0}\cdots\subst{c_{n-1}}{x_{n-1}}$ for each $\varphi \in \Phi$, and abbreviate this long substitution with the notation $\varphi\subst{\vec c}{\vec x}$. Thus we also have $\varphi = \varphi^\Phi\subst{\vec x}{\vec c}$.

The above translation is harmless when it comes to provability, as shown by the following lemma.
\begin{lemma}
\label{lem:noconstants}
  Let $\Phi$ be a finite set of $\QRC$ formulas such that $\varphi, \psi \in \Phi$. Then:
  \begin{equation*}
    \varphi \vdash \psi
    \iff
    \varphi^\Phi \vdash \psi^\Phi
    .
  \end{equation*}
\end{lemma}
\begin{proof}
  If $\varphi \vdash \psi$, then $\varphi^\Phi\subst{\vec x}{\vec c} \vdash \psi^\Phi\subst{\vec x}{\vec c}$, and since no constant appears in either $\varphi^\Phi$ nor $\psi^\Phi$, we can iterate Rule~\ref{rule:constants} to obtain $\varphi^\Phi \vdash \psi^\Phi$.
  If, on the other hand, we have $\varphi^\Phi \vdash \psi^\Phi$, then we can repeatedly use Rule~\ref{rule:term_instantiation} to obtain $\varphi^\Phi\subst{\vec x}{\vec c} \vdash \psi^\Phi\subst{\vec x}{\vec c}$, which is $\varphi \vdash \psi$.
\end{proof}

The semantics presented here for $\QRC$ is a generalization of more traditional semantics for quantified modal logics, since it uses maps between domains. However, for this section, we will only need constant domain models. These can be easily interpreted in the usual framework as described in \cite{HughesCresswell1996}. The following is a well-known (and easily provable) result.

\begin{lemma}[Soundness for $\QS$, \cite{HughesCresswell1996}]
\label{lem:QSsoundness}
  If $\F = \la W, R \ra$ satisfies every theorem of $\S$, then any constant domain model based on $\F$ satisfies every theorem of $\QS$.
\end{lemma}

Thus any transitive constant domain model is a $\QKf$ model, and any transitive and conversely well-founded constant domain model is a $\QGL$ model.\footnote{We assume constant domain models for simplicity, but these results also hold in the inclusive case.}

We are now ready to show the main theorem of this section.

\begin{theorem}
\label{thm:positive_fragment}
  Let $\varphi$ and $\psi$ be $\QRC$ formulas and let $\QS$ be any logic between $\QKf$ and $\QGL$. Then $\varphi \vdash \psi$ if and only if there is a finite set $\Phi$ such that $\varphi, \psi \in \Phi$ and $\QS \vdash \varphi^\Phi \to \psi^\Phi$.
\end{theorem}
\begin{proof}
  For the left-to-right implication, let $\Phi$ be the set of formulas appearing in the proof of $\varphi \vdash \psi$. We proceed to show $\QKf \vdash \varphi^\Phi \to \psi^\Phi$ (implying $\QS \vdash \varphi^\Phi \to \psi^\Phi$) by induction on the length of the $\QRC$ proof. The axioms $\varphi \vdash \top$, $\varphi \vdash \varphi$, and $\varphi \land \psi \vdash \varphi$ and the conjunction introduction and cut rules are straightforward.

  For the $\QRC$ necessitation rule, assume that $\varphi \vdash \psi$, which by the induction hypothesis gives us $\QKf \vdash \varphi^\Phi \to \psi^\Phi$. Then by taking the contrapositive and then applying the $\QKf$ necessitation rule, we obtain $\QKf \vdash \Box (\neg \psi^\Phi \to \neg \varphi^\Phi)$. This implies $\QKf \vdash \Box \neg \psi^\Phi \to \Box \neg \varphi^\Phi$, and taking the contrapositive again gives us $\QKf \vdash \Diamond \varphi^\Phi \to \Diamond \psi^\Phi$, as desired.

  The translation of the $\QRC$ transitivity axiom, $\Diamond \Diamond \varphi^\Phi \to \Diamond \varphi^\Phi$, is just the contrapositive of an instance of the $4$ axiom $\Box A \to \Box \Box A$.

  We turn to the $\forall$-introduction rule on the right. If $\varphi \vdash \psi$ and $x \notin \fv(\varphi)$, then by the induction hypothesis $\QKf \vdash \varphi^\Phi \to \psi^\Phi$ and we can be assured that $x \notin \fv(\varphi^\Phi)$ because all the new variables introduced by the $\cdot^\Phi$ translation are fresh with respect to $\Forall{x} \psi$, and are consequently different from $x$. Then $\QKf \vdash \varphi^\Phi \to \Forall{x} \psi^\Phi$ by Rule~\ref{rule:QforallI}.

  In the case of the $\forall$-introduction rule on the left, suppose that $\varphi\subst{x}{t} \vdash \psi$, with $t$ free for $x$ in $\varphi$.
  Defining $t^\Phi$ as either $t$ itself (when $t$ is a variable), or whichever variable replaces $t$ by the $\cdot^\Phi$ translation (when $t$ is a constant), observe that $(\varphi\subst{x}{t})^\Phi$ is $\varphi^\Phi\subst{x}{t^\Phi}$.
  Thus we have $\QKf \vdash \varphi^\Phi\subst{x}{t^\Phi} \to \psi^\Phi$ by the induction hypothesis, and we wish to show $\QKf \vdash \Forall{x} \varphi^\Phi \to \psi^\Phi$. 
  We know $\QKf \vdash \Forall{x} \varphi^\Phi \to \varphi^\Phi\subst{x}{t^\Phi}$ by Axiom~\ref{ax:QforallE}, and so we are done.

  For the term instantiation rule, suppose that $\varphi \vdash \psi$ and let $x$ and $t$ be such that $t$ is free for $x$ in both $\varphi$ and $\psi$. By the induction hypothesis we have established $\QKf \vdash \varphi^\Phi \to \psi^\Phi$. Using Rule~\ref{rule:QforallI} with $\top$ as the antecedent, we obtain $\QKf \vdash \Forall{x} (\varphi^\Phi \to \psi^\Phi)$, and then we may conclude $\QKf \vdash \varphi^\Phi\subst{x}{t^\Phi} \to \psi^\Phi\subst{x}{t^\Phi}$ by Axiom~\ref{ax:QforallE}, as desired.

  Finally, in the case of the constant elimination rule, if $\varphi\subst{x}{c} \vdash \psi\subst{x}{c}$, then $\QKf \vdash \varphi^\Phi\subst{x}{c^\Phi} \to \psi^\Phi\subst{x}{c^\Phi}$ by the induction hypothesis. Since $c^\Phi$ is a variable, we can generalize it through Rule~\ref{rule:QforallI}, obtaining $\QKf \vdash \Forall{c^\Phi} (\varphi^\Phi\subst{x}{c^\Phi} \to \psi^\Phi\subst{x}{c^\Phi})$. We then use Axiom~\ref{ax:QforallE} to instantiate $c^\Phi$ with $x$, obtaining $\QKf \vdash \varphi^\Phi\subst{x}{c^\Phi}\subst{c^\Phi}{x} \to \psi^\Phi\subst{x}{c^\Phi}\subst{c^\Phi}{x}$.
  Since $c$ does not appear in either $\varphi$ nor $\psi$, we also know that $c^\Phi$ does not appear in either $\varphi^\Phi$ nor $\psi^\Phi$, and so $\varphi^\Phi\subst{x}{c^\Phi}\subst{c^\Phi}{x} = \varphi^\Phi$, and similarly for $\psi$. We conclude that $\QKf \vdash \varphi^\Phi \to \psi^\Phi$, as desired.

  For the right-to-left implication, we prove the contrapositive. If $\varphi \not\vdash \psi$, then by Lemma~\ref{lem:noconstants} we know that $\varphi^\Phi \not\vdash \psi^\Phi$ for any $\Phi$ including both $\varphi$ and $\psi$. Let $\M$ be a $\QRC$ model satisfying $\varphi^\Phi$ and not satisfying $\psi^\Phi$, as given by Theorem~\ref{thm:constant_domain}. This model is transitive, irreflexive, and has finitely-many worlds, which means it is a $\QGL$ model by Lemma~\ref{lem:QSsoundness}. We conclude that $\QGL \not\vdash \varphi^\Phi \to \psi^\Phi$, and consequently that $\QS \not\vdash \varphi^\Phi \to \psi^\Phi$.
\end{proof}

We end with the following observation. The proof of Theorem~\ref{thm:positive_fragment} holds for any quantified modal logic that extends $\QKf$ and is sound for any combination of constant-domain, finite, transitive and irreflexive models. The Barcan Formula ($\mathsf{BF}$), $\Forall{x} \Box A \to \Box \Forall{x} A$, is popular in the world of quantified modal logic and is always sound in constant-domain models. Thus the proof of Theorem~\ref{thm:positive_fragment} also serves to see that, for example, $\varphi \vdash \psi$ if and only if $\QKf + \mathsf{BF} \vdash \varphi^\Phi \to \psi^\Phi$.
  This is curious because $\mathsf{BF}$ does not hold in $\QGL$ and is in fact unsound with the usual provability interpretation of $\Box$.
  From this we conclude that it would be worthwhile to study $\QRC$ from points of view unrelated to provability.

\section{Arithmetical semantics}
\label{sec:arithmetics}

Recall that the language of arithmetic is that of first-order logic together with the symbols $\{ 0, 1, +, \times, \leq, = \}$ with their usual arities (see \cite{HajekPudlak:1993:Metamathematics} for details). We will rely on $\Sigma^0_1$ collection throughout the sections on arithmetic, so we take $\isig{1}$ as our base theory.

Let $T$ be an elementary presented extension of $\isig{1}$, meaning that there is a bounded formula $\Ax_T(u)$ that is true in the standard model if and only if $u$ is (the Gödel number of) an axiom of $T$. In this setting it is traditional to define Gödel's provability predicate $\Box_T \varphi$ as $\Exists{p} \Prf_T(p, \gnum{\varphi})$, where
\begin{align*}
  \Prf_T(p, n) :=
    \text{sequence}(p)
    &\land p_{|p|-1} = n \land {}\\
    \Forall{k < |p|} &(
      \Ax_T(p_k)
      \lor \Exists{i, j < k} \text{MP}(p_i, p_j, p_k)
      \lor \Exists{i < k} \text{Gen}(p_i, p_k)
    )
    .
\end{align*}

Roughly, $\Box_T \varphi$ formalizes that there is a Hilbert-style proof of $\varphi$, that is, a finite sequence $p_0, \hdots, p_{m}$ such that $p_{m}$ is $\varphi$ and that each $p_k$ is either (the Gödel number of) an axiom of $T$ or follows from previous elements of the sequence through either \emph{modus ponens} or generalization.

Note that all the functions used in the definition of $\Prf_T$ can be naturally defined as bounded formulas in $\isig{1}$, and thus $\Prf_T$ is itself a bounded formula. This means that $\Box_T \varphi$ is $\Sigma^0_1$. We write $\Diamond_T \varphi$ as shorthand for $\neg \Box_T \neg \varphi$.

In what follows we will be interested in c.e.~theories, which are theories whose axioms can be defined by a $\Sigma^0_1$ formula. It is known by Craig's Trick \cite{Craig1953, Feferman1960} that any such theory has an equivalent elementary presentation, allowing us to use the regular definition of $\Box_T$. However, we will work with specific axiomatizations and thus it is sometimes more convenient to allow $\Prf$ to be a $\Sigma^0_1$ formula. For a given $\Sigma^0_1$ axiomatization $\tau$ of $T$, we define $\Box_\tau \varphi := \Exists{p} \Prf_\tau(p, \gnum{\varphi})$, where
\begin{align*}
  \Prf_\tau(p, n) :=
    \text{sequence}(p)
    &\land p_{|p|-1} = n \land {}\\
    \Forall{k < |p|} &(
      \tau(p_k)
      \lor \Exists{i, j < k} \text{MP}(p_i, p_j, p_k)
      \lor \Exists{i < k} \text{Gen}(p_i, p_k)
    )
    ,
\end{align*}
the only difference from $\Prf_T$ being that we use the $\Sigma^0_1$ formula $\tau$ instead of the bounded $\Ax_T$.
Thus $\Prf_\tau$ is equivalent to a $\Sigma^0_1$ formula (provably in $\isig{1}$), and so $\Box_\tau \varphi$ is still $\Sigma^0_1$. We define $\Diamond_\tau \varphi$ as $\neg \Box_\tau \neg \varphi$, as usual.

We wish to interpret the strictly positive formulas in the language of $\QRC$ as parameterized axiomatizations of arithmetical theories extending $\isig{1}$. Let $\tau$ be a bounded axiomatization of a sound c.e.~base theory $T$ extending $\isig{1}$.

  A realization $\cdot\is$ interprets each $n$-ary relation symbol $S(\vec{x}, \vec{c})$ as an $(n+1)$-ary $\Sigma^0_1$ formula $\sigma(u, \vec{y},\allowbreak \vec{z})$ in the language of arithmetic, where $\vec{y}$ matches with $\vec{x}$ and $\vec{z}$ matches with $\vec{c}$.\footnote{In the sections on arithmetic, we always use $x$ for variables of $\QRC$ and $y, z, u$ for variables of arithmetic. Furthermore, $u$ is reserved for the Gödel numbers of axioms of theories.}
  We then extend this notion to any formula as follows (we add $\tau(u)$ to the axioms of the interpretation of any relation symbol to guarantee that every theory is an extension of $T$).
    
\begin{definition}[\cite{QRC1}]
\label{def:isT}

Let $\cdot\is$ be a realization such that, for a given $n$-ary predicate $S$ and terms $\vec t$, $S(\vec t)\is$ is an $(n+1)$-ary $\Sigma^0_1$ arithmetical formula where modal variables $x_k$ are interpreted as $y_k$ and modal constants $c_k$ are interpreted as $z_k$. We extend this realization to $\QRC$ formulas as follows:
  \begin{itemize}
    \item $\top\isT := \tau(u)$;
    \item $S(\vec{x}, \vec{c})\isT : = S(\vec{x},\vec{c})\is \lor \tau(u)$;
    \item $(\psi \land \delta)\isT := \psi\isT \lor \delta\isT$;
    \item $(\lozenge \psi)\isT := \tau(u) \lor (u = \gnum{\Diamond_{\psi\isT} \top})$;
    \item $(\Forall{x_i} \psi)\isT := \Exists{y_i} \psi\isT$.
  \end{itemize}
\end{definition}
We remark that, if the free variables of $\varphi$ are $x_0, \hdots, x_{n-1}$ and the constants appearing in $\varphi$ are $c_0, \hdots, c_{m-1}$, then the free variables of $\varphi\isT$ are $u, y_0, \hdots, y_{n-1}, z_0, \hdots, z_{m-1}$.
For example, if $S(x_0, c_1)\isT = \sigma(u, y_0, z_1) \lor \tau(u)$, then
\begin{equation*}
  (\Diamond S(x_0, c_1))\isT =
    \tau(u) \lor (u = \gnum{\Diamond_{\sigma(u, \dot{y_0}, \dot{z_1}) \lor \tau(u)} \top})
  .
\end{equation*}
The dotted variables $\dot{y_0}$ and $\dot{z_1}$ appearing in the expression $u = \gnum{\Diamond_{\sigma(u, \dot{y_0}, \dot{z_1}) \lor \tau(u)} \top}$ indicate that $y_0$ and $z_1$ are free variables of this expression that, upon being instantiated by natural numbers $n$ and $m$, shall be replaced by the numerals $\bar{n}$ and $\bar{m}$ instead.

Since we have $\Sigma^0_1$ collection, $\varphi\isT$ is always provably equivalent to a $\Sigma^0_1$ formula, and it represents the axiomatization of an extension of $T$.

We briefly inspect our intuitions regarding this arithmetic interpretation and the interaction between $\forall$ and $\Diamond$. As mentioned in Lemma~\ref{lem:QRC1consequences}, the sequent $\Diamond \Forall{x} \varphi \vdash \Forall{x} \Diamond \varphi$ is provable in $\QRC$. In the arithmetical reading it says that if the union of many theories (the $\Forall{x} \varphi$ part) is consistent, then so is each individual theory. The converse, on the other hand, must not be a consequence of $\QRC$, for it is not in general the case that the union of many consistent theories is itself consistent.

We can now define $\QRCt(T)$:
  \begin{equation*} 
  \QRCt(T) := \{\varphi(\vec{x}, \vec{c}) \vdash \psi(\vec{x}, \vec{c}) \mid
    \Forall{\cdot\is} 
    T \vdash \Forall{\theta} \Forall{\vec{y}, \vec{z}} (\Box_{\psi\isT} \theta \to \Box_{\varphi\isT} \theta)
  \}
  ,
  \end{equation*}
where $\theta$ is a closed formula and $\varphi\isT, \psi\isT$ in general depend on $\vec{y}$ and $\vec{z}$.

 We will show that
 \begin{equation*}
   \QRC = \QRCt(T)
 \end{equation*}
 for any sound c.e.~theory $T$ extending $\isig{1}$. The left-to-right inclusion was already proved in \cite{QRC1} for $\isig{1}$, and the same proof goes through for any such $T$. The most notable part of the proof is the case of Rule~\ref{rule:forallR} ($\forall$ introduction on the right), which uses $\Sigma^0_1$ collection. We go into detail on the very similar proof of Theorem~\ref{thm:HAsoundness}.

\begin{theorem}[Arithmetical soundness, \cite{QRC1}]
$\QRC \subseteq \QRCt(T)$.
\end{theorem}

We dedicate the next section to the proof of the other inclusion.

\section{Arithmetical completeness}
\label{sec:arithmetical_completeness}

The proof of arithmetical completeness closely follows the proof of Solovay's Theorem found in \cite{Boolos:1993:LogicOfProvability}. The arithmetical semantics used there is different from the one presented in the previous section, allowing only for finite extensions of the base theory instead of arbitrary c.e.~ones. This is not a problem because finite axiomatizations are enough to show completeness. We use $\cdot\is$ in the general setting and $\cdot\fs$ in the finite one. The base theory $T$ considered here is any sound c.e.~extension of $\isig{1}$. Note that we use the bounded $\Prf_T$ formula in the definition of $\cdot\fsT$ (depending on an elementary axiomatization $\tau$ of $T$), as otherwise Solovay's proof wouldn't go through.

\begin{definition}[$\cdot\fsT$]
\label{def:fsT}

Let $\cdot\fs$ be a realization such that, for a given predicate $S$ and terms $\vec t$, $S(\vec t)\fs$ is an arithmetical formula with the same arity as $S$ where modal variables $x_k$ are interpreted as $y_k$ and modal constants $c_k$ are interpreted as $z_k$. We extend this realization to $\QRC$ formulas as follows:
\begin{itemize}
  \item $\top\fsT := \top$;
  \item $S(\vec{t})\fsT := S(\vec t)\fs$;
  \item $(\varphi \land \psi)\fsT := \varphi\fsT \land \psi\fsT$;
  \item $(\Diamond \varphi)\fsT := \Diamond_T \varphi\fsT$;
  \item $(\Forall{x_k} \varphi)\fsT := \Forall{y_k} \varphi\fsT$.
\end{itemize}
\end{definition}

The idea of Solovay's proof is to embed a Kripke model not satisfying the desired unprovable formula in the language of arithmetic. If the embedding is done correctly, it is possible to prove that a formula is satisfied at a world of the Kripke model exactly when it is a consequence of the representation of that world in the desired theory $T$.

Given two strictly positive formulas $\varphi$ and $\psi$ such that $\varphi \not\vdash \psi$ in $\QRC$, let $\M_{\varphi, \psi}$ be a finite, irreflexive, and constant domain adequate model satisfying $\varphi$ and not satisfying $\psi$ at the root $1$ under a $1$-assignment $g_{\varphi, \psi}$. This model and assignment exist by Theorem~\ref{thm:constant_domain}. Since $\M_{\varphi, \psi}$ has constant domain, we refer to $g_{\varphi, \psi}$ and any other $i$-assignments as just assignments, omitting the relevant world.

We assume that the worlds of $\M_{\varphi, \psi}$ are $W = \{ 1, 2, \ldots , N \}$, where $1$ is the root. We define a new (adequate) model $\M$ that is a copy of $\M_{\varphi, \psi}$, except that it has an extra world $0$ as the new root. This world $0$ is connected to all the other worlds through $R$ and has the same domain, constant interpretation, and relational symbol interpretation as $1$. The functions $\eta_{0, i}$ with $0 < i \leq N$ are all defined as the identity function.

Let $\lambda_i$ be the Solovay sentences for $T$ as defined in \cite{Boolos:1993:LogicOfProvability}, satisfying the following Embedding Lemma.

\begin{lemma}[Embedding, \cite{Boolos:1993:LogicOfProvability}]
\label{lem:embedding}
\
\begin{enumerate}[label=\upshape(\arabic*),ref=\thetheorem.(\arabic*)]
  \item\label{embedding:Lambda0istrue}
    $\N \vDash \lambda_0$;

  \item\label{embedding:SomeOfAll}
    $T \vdash \bigvee_{i \leq N} \lambda_i$;

  \item\label{embedding:OneOfAll}
    $T \vdash \lambda_i \to \bigwedge_{j \leq N, j \neq i} \neg \lambda_j$, for $i \leq N$;

  \item\label{embedding:Diamond}
    $T \vdash \lambda_i \to \bigwedge_{j \leq N, iRj} \Diamond_T \lambda_j$, for $i \leq N$;

  \item\label{embedding:Box}
    $T \vdash \lambda_i \to  \Box_T \bigvee_{j \leq N, iRj} \lambda_j$, for $0 < i \leq N$.
\end{enumerate}
\end{lemma}

The domain of every world is $M$. Let $m$ be the size of $M$, and $\bij{\cdot}$ be a bijection between $M$ and the set of numerals $\{\overline{0}, \hdots, \overline{m - 1}\}$.

We now define for a given $n$-ary predicate symbol $S$ and terms $\vec{t} = t_0, \hdots, t_{n-1}$ the $\cdot\ofs$ interpretation as follows:
  \begin{gather*}
    S(\vec{t})\ofs :=
    \bigvee_{i \leq N} \Big(\lambda_i \land \Phi^{S(\vec{t})}_i \Big) \text{, where}\\
    \Phi_i^{S(\vec{t})} := 
          \bigvee_{\la a_0, \hdots, a_{n-1} \ra \in S^{J_i}}
          \bigwedge_{l < n} \Big(
            \bij{a_l} = 
              \begin{cases}
                y_k \umod m &\text{if } t_l = x_k\\
                z_k \umod m &\text{if } t_l = c_k
              \end{cases}
          \Big)
      .
  \end{gather*}
Here $x_k$ is any $\QRC$ variable, $y_k$ is the corresponding $T$ variable, $c_k$ is any $\QRC$ constant, $z_k$ is the corresponding $T$ variable, and $\cdot^{J_i}$ is the denotation of a predicate symbol at world $i$.
Note that $x_k \in \fv(S(\vec{t}))$ if and only if $y_k \in \fv(S(\vec{t})\ofs)$ and $c_k$ appears in $S(\vec t)$ if and only if $z_k \in \fv(S(\vec t)\ofs)$. The $\cdot\ofs$ interpretation is extended to non-predicate formulas as described for $\cdot\fs$ at the beginning of this section.

\begin{remark}
  For any $\QRC$ formula $\varphi$, variable $x_k$ and constant $c_k$, we have that $x_k \in \fv(\varphi)$ if and only if $y_k \in \fv(\varphi\ofsT)$ and that $c_k$ appears in $\varphi$ if and only if $z_k \in \fv(\varphi\ofsT)$.
\end{remark}

We further note that $\varphi\ofsT$ is invariant under replacing variables by themselves modulo $m$, provably in $T$.

\begin{lemma}
\label{lem:mod}
  For any $\QRC$ formula $\varphi$ and any arithmetical variable $y$:
  \begin{enumerate}
    \item
      $
        T \vdash
          \varphi\ofsT \leftrightarrow \varphi\ofsT\subst{y}{y \umod m}
      $;
    \item
      $
        T \vdash
          \Forall{y} \varphi\ofsT
          \leftrightarrow
          \Forall{y < m} \varphi\ofsT
      $.
  \end{enumerate}
\end{lemma}
\begin{proof}
  The second item is a straightforward consequence of the first, which we prove by external induction on the complexity of $\varphi$. The cases of $\top$ and $\land$ are easy.

  For the case of the relation symbols, consider without loss of generality $S(x_0, c_0)$. The formula $S(x_0, c_0)\ofs$ has two free variables, namely $y_0$ and $z_0$, so the result is trivial when $y$ is not one of these.

  We check first that $T \vdash S(x_0, c_0)\ofs \leftrightarrow S(x_0, c_0)\ofs\subst{y_0}{y_0 \umod m}$. We have
  \begin{equation*}
    S(x_0, c_0)\ofs =
      \bigvee_{j \leq N}\Big(
        \lambda_j \land
          \bigvee_{\la a_0, a_1 \ra \in S^{J_j}}
          (
            \bij{a_0} = y_0 \umod m
            \land
            \bij{a_1} = z_0 \umod m
          )
      \Big)
      ,
  \end{equation*}
  and hence, noting that $y_0$ and $z_0$ are different variables,
  \begin{align*}
    S(x_0, c_0)\ofs&\subst{y_0}{y_0 \umod m} =\\
      &\bigvee_{j \leq N}\Big(
        \lambda_j \land
          \bigvee_{\la a_0, a_1 \ra \in S^{J_j}}
          (
            \bij{a_0} = (y_0 \umod m) \umod m
            \land
            \bij{a_1} = z_0 \umod m
          )
      \Big)
      .
  \end{align*}
  These are equivalent because $T$ proves $(y_0 \umod m) \umod m = y_0 \umod m$.
  If $y$ is $z_0$ instead, the argument is analogous.

  Consider now the case of $\Forall{x_0} \varphi$. If $y$ is $y_0$ there is nothing to show ($y_0$ is not a free variable of $(\Forall{x_0} \varphi)\ofsT$). Thus we may assume that $(\Forall{x_0} \varphi)\ofsT\subst{y}{y \umod m}$ is the same as $\Forall{y_0} \varphi\ofsT\subst{y}{y \umod m}$. By the induction hypothesis, $T \vdash \varphi\ofsT \leftrightarrow \varphi\ofsT\subst{y}{y \umod m}$, from which we obtain $T \vdash \Forall{y_0} \varphi\ofsT \leftrightarrow \Forall{y_0} \varphi\ofsT\subst{y}{y \umod m}$, as desired.

  Finally, in the case of $\Diamond \varphi$, note that $(\Diamond \varphi)\ofsT\subst{y}{y \umod m}$ is $\Diamond_T \varphi\ofsT\subst{y}{y \umod m}$. Thus the result is straightforward from the induction hypothesis under the box.
\end{proof}

We now prove two versions of a Truth Lemma, one for when $\varphi$ is forced at a world $i$ of $\M$, and one for when it isn't. We wish to show that $T$ proves $\varphi\ofsT$ (respectively $\neg \varphi\ofsT$) as a consequence of $\lambda_i$, as long as the free variables of $\varphi$ are interpreted in the same way in both settings.

In order to use concise notation, we shall say $\vec{y}$ instead of $y_0, \hdots, y_{n-1}$, and similarly for other terms. We also abbreviate iterated substitutions, writing $\varphi\ofsT\subst{\vec{y}}{\vec{\bij{g(x)}}}$ instead of $\varphi\ofsT\subst{y_0}{\bij{g(x_0)}}\cdots\subst{y_{n-1}}{\bij{g(x_{n-1})}}$ and writing $\varphi\ofsT\subst{\vec z}{\vec{\bij{c^{I_0}}}}$ instead of $\varphi\ofsT\subst{z_0}{\bij{(c_0)^{I_0}}}\cdots\subst{z_{n'-1}}{\bij{(c_{n'-1})^{I_0}}}$.

\begin{lemma}[Truth Lemma: positive]
\label{lem:truth_positive}
  For any $\QRC$ formula $\varphi$ with free variables $\vec x = x_0, \hdots, x_{n-1}$ and constants $\vec c = c_0, \hdots, c_{n'-1}$, any world $i \leq N$, and any assignment $g$:
  \begin{equation*}
    \M, i \Vdash^g \varphi
    \implies
    T \vdash \lambda_i \to
      \varphi\ofsT
        \subst{\vec y}{\vec{\bij{g(x)}}}
        \subst{\vec z}{\vec{\bij{c^{I_0}}}}
    .
  \end{equation*}
\end{lemma}
\begin{proof}
  By external induction on the complexity of $\varphi$. The cases of $\top$ and $\land$ are straightforward.

  In the case of relational symbols, we assume without loss of generality that the relevant formula is $S(x_0, c_0)$.
  If $\M, i \Vdash^g S(x_0, c_0)$, then $\la g(x_0), (c_0)^{I_i} \ra \in S^{J_i}$.
  Reason in T and assume $\lambda_i$.
  It suffices to prove $\Phi^{S(x_0, c_0)}_i\subst{y_0}{\bij{g(x_0)}}\subst{z_0}{\bij{(c_0)^{I_0}}}$, which implies $S(x_0, c_0)\ofs\subst{y_0}{\bij{g(x_0)}}\subst{z_0}{\bij{(c_0)^{I_0}}}$ under the assumption of $\lambda_i$.
  We need to find $\la a_0, a_1 \ra \in S^{J_i}$ such that $\bij{a_0} = \bij{g(x_0)} \umod m$ and $\bij{a_1} = \bij{(c_0)^{I_0}} \umod m$.
  Noting that $\bij{b} \umod m$ is provably equal to $\bij{b}$ for any $b$, we pick $a_0 := g(x_0)$ and $a_1 := (c_0)^{I_0}$. This concludes this step of the proof because $(c_0)^{I_0}$ is equal to $(c_0)^{I_i}$, for any world $i$.

  For $\Forall{x_0} \varphi$, assume without loss of generality that the free variables of $\varphi$ are $x_0$ and $x_1$. If $\M, i \Vdash^g \Forall{x_0} \varphi$ then for every assignment $h \xaltern{x_0} g$ we have $\M, i \Vdash^h \varphi$. We wish to show
  \begin{equation*}
    T \vdash \lambda_i \to (\Forall{y_0} \varphi\ofsT)
                                          \subst{y_1}{\bij{g(x_1)}}
                                          \subst{\vec z}{\vec{\bij{c^{I_0}}}}
    .
  \end{equation*}
  Reason in $T$ and assume $\lambda_i$. By Lemma~\ref{lem:mod}, it is enough to show
  \begin{equation*}
    (\Forall{y_0 < m} \varphi\ofsT)
      \subst{y_1}{\bij{g(x_1)}}
      \subst{\vec z}{\vec{\bij{c^{I_0}}}}
    .
  \end{equation*}
  Since $y_0$, $y_1$, and $\vec z$ are all different, we can push the substitutions inside and prove
  \begin{equation*}
    \Forall{y_0 < m} \varphi\ofsT
      \subst{y_1}{\bij{g(x_1)}}
      \subst{\vec z}{\vec{\bij{c^{I_0}}}}
  \end{equation*}
instead.
  Let $y_0 < m$ be arbitrary.
  Since $\bij{\cdot}$ is a bijection, there is $a \in M$ such that $\bij{a} = y_0$. We define an assignment $h$ such that $h \xaltern{x_0} g$ and $h(x_0) := a$. By assumption, $\M, i \Vdash^h \varphi$, so by the induction hypothesis we obtain
  \begin{equation*}
    \varphi\ofsT
      \subst{y_0}{\bij{h(x_0)}}
      \subst{y_1}{\bij{h(x_1)}}
      \subst{\vec z}{\vec{\bij{c^{I_0}}}}
    .
  \end{equation*}
  This concludes the argument because $\bij{h(x_0)} = y_0$ and $h(x_1) = g(x_1)$.

  Finally, consider the case of $\Diamond \varphi$. If $\M, i \Vdash^g \Diamond \varphi$, then there is $j \leq N$ such that $iRj$ and $\M, j \Vdash^g \varphi$. Reason in $T$ and assume $\lambda_i$. By Lemma~\ref{embedding:Diamond}, we obtain $\Diamond_T \lambda_j$. Then the induction hypothesis under the box gives us $\Diamond_T (\varphi\ofsT\subst{\vec y}{\vec{\bij{g(x)}}}\subst{\vec z}{\vec{\bij{c^{I_0}}}})$, as desired.
\end{proof}

\begin{lemma}[Truth Lemma: negative]
\label{lem:truth_negative}
For any $\QRC$ formula $\varphi$ with free variables $\vec x = x_0, \hdots, x_{n-1}$ and constants $\vec c = c_0, \hdots, c_{n'-1}$, any world $0 < i \leq N$, and any assignment $g$:
  \begin{equation*}
    \M, i \not\Vdash^g \varphi
    \implies
    T \vdash \lambda_i \to
      \neg \varphi\ofsT
        \subst{\vec y}{\vec{\bij{g(x)}}}
        \subst{\vec z}{\vec{\bij{c^{I_0}}}}
    .
  \end{equation*}
\end{lemma}
\begin{proof}
  By external induction on the complexity of $\varphi$. The cases of $\top$ and $\land$ are straightforward.

  For the relational symbols, consider $S(x_0, c_0)$ without loss of generality.
  If $\M, i \not\Vdash^g S(x_0, c_0)$, then $\la g(x_0), (c_0)^{I_i} \ra \notin S^{J_i}$.
  Reason in $T$ and assume $\lambda_i$.
  We obtain $\neg \lambda_j$ for every $j \neq i$ by Lemma~\ref{embedding:OneOfAll}, and hence need only show $\neg \Phi^{S(x_0, c_0)}_i\subst{y_0}{\bij{g(x_0)}}\subst{z_0}{\bij{(c_0)^{I_0}}}$.
  In other words, we need to check that if $\la a_0, a_1 \ra \in S^{J_i}$, either $\bij{a_0} \neq \bij{g(x_0)} \umod m$, or $\bij{a_1} \neq \bij{(c_0)^{I_0}} \umod m$.
  This follows from our observation that $\la g(x_0), (c_0)^{I_i} \ra \notin S^{J_i}$, taking into account that $\bij{b} \umod m$ is equal to $\bij{b}$ for any $b$, that $\bij{\cdot}$ is injective, and that $(c_0)^{I_0} = (c_0)^{I_i}$.

  Consider now the case of $\Forall{x_0} \varphi$.
  We assume without loss of generality that $\fv(\varphi) = \{x_0, x_1\}$.
  If $\M, i \not\Vdash^g \Forall{x_0} \varphi$, then there is an assignment $h \xaltern{x_0} g$ such that $\M, i \not\Vdash^h \varphi$.
  Reason in $T$ and assume $\lambda_i$.
  We obtain $\neg \varphi\ofsT\subst{y_0}{\bij{h(x_0)}}\subst{y_1}{\bij{h(x_1)}}\subst{\vec z}{\vec{\bij{c^{I_0}}}}$ by the induction hypothesis, which implies $\neg \Forall{y_0} \varphi\ofsT\subst{y_1}{\bij{h(x_1)}}\subst{\vec z}{\vec{\bij{c^{I_0}}}}$.
  This is what we wanted, taking into account that $h(x_1) = g(x_1)$.

  Finally, in the case of $\Diamond \varphi$ with $x_0$ as the only free variable (without loss of generality), assume that $\M, i \not\Vdash^g \Diamond \varphi$. Then for every $j$ such that $iRj$, we have $\M, j \not\Vdash^g \varphi$ and thus for each such $j$ the induction hypothesis gives us $T \vdash \lambda_j \to \neg\varphi\ofsT\subst{y_0}{\bij{g(x_0)}}$, which put together implies $T \vdash \bigvee_{iRj} \lambda_j \to \neg \varphi\ofsT\subst{y_0}{\bij{g(x_0)}}$. Reason in $T$ and assume $\lambda_i$. By Lemma~\ref{embedding:Box} and our assumption, we obtain $\Box_T \bigvee_{iRj} \lambda_j$. Taking the previous observation under the box, we conclude $\Box_T \neg \varphi\ofsT\subst{y_0}{\bij{g(x_0)}}$, which is precisely $\neg (\Diamond \varphi)\ofsT\subst{y_0}{\bij{g(x_0)}}$.
\end{proof}

We are ready to prove something analogous to Solovay's Theorem, which is the precursor to our desired completeness theorem.

\begin{theorem}
\label{thm:arithmetical_completeness*}

  If $\varphi, \psi$ are $\QRC$ formulas with free variables $\vec{x}$ and constants $\vec c$ such that $\varphi \not\vdash \psi$, we have
  $T \not\vdash (\varphi\ofsT \to \psi\ofsT)\subst{\vec{y}}{\vec{\bij{g_{\varphi, \psi}(x)}}}\subst{\vec z}{\vec{\bij{c^{I_0}}}}$.
\end{theorem}
\begin{proof}
  Recall that $\M$ satisfies $\varphi$ and not $\psi$ at world $1$ under the assignment $g := g_{\varphi, \psi}$.
  
  Since $\M, 1 \Vdash^g \varphi$, we obtain $T \vdash \lambda_1 \to \varphi\ofsT\subst{\vec{y}}{\vec{\bij{g(x)}}}\subst{\vec z}{\vec{\bij{c^{I_0}}}}$ from Lemma~\ref{lem:truth_positive}.
  Since $\M, 1 \not\Vdash^g \psi$, we obtain $T \vdash \lambda_1 \to \neg \psi\ofsT\subst{\vec{y}}{\vec{\bij{g(x)}}}\subst{\vec z}{\vec{\bij{c^{I_0}}}}$ from Lemma~\ref{lem:truth_negative}.
  Thus $T \vdash \lambda_1 \to \neg (\varphi\ofsT \to \psi\ofsT)\subst{\vec{y}}{\vec{\bij{g(x)}}}\subst{\vec z}{\vec{\bij{c^{I_0}}}}$.

  By Lemma~\ref{embedding:Diamond} and the fact that $0R1$, we obtain $T \vdash \lambda_0 \to \Diamond_T \lambda_1$, so putting this together with the previous observation under the box, $T \vdash \lambda_0 \to \Diamond_T \neg (\varphi\ofsT \to \psi\ofsT)\subst{\vec{y}}{\vec{\bij{g(x)}}}\subst{\vec z}{\vec{\bij{c^{I_0}}}}$.

  By Lemma~\ref{embedding:Lambda0istrue}, we know that $\mathbb{N} \vDash \lambda_0$, and thus by the soundness of $T$, we know that $\mathbb{N} \vDash \Diamond_T \neg (\varphi\ofsT \to \psi\ofsT)\subst{\vec{y}}{\vec{\bij{g(x)}}}\subst{\vec z}{\vec{\bij{c^{I_0}}}}$. Then it must be the case that $T \not\vdash (\varphi\ofsT \to \psi\ofsT)\subst{\vec{y}}{\vec{\bij{g(x)}}}\subst{\vec z}{\vec{\bij{c^{I_0}}}}$.
\end{proof}

We now define an arithmetical realization $\cdot\ois$ in the style of Section~\ref{sec:arithmetics} that behaves like $\cdot\ofs$. Recall that if $\tau$ is a $\Sigma^0_1$ formula axiomatizing $T$, the realization $\cdot\isT$ sends $\QRC$ formulas to $\Sigma^0_1$ axiomatizations of theories extending $T$.
So in particular we always have $(\varphi \land \psi)\isT = \varphi\isT \lor \psi\isT$, because $\chi$ is an axiom of the union of the theories axiomatized by $\varphi\isT$ and $\psi\isT$ precisely when $\chi$ is an axiom of one of them.
This is different from $\cdot\fsT$ realizations such as $\cdot\ofsT$, which send $\QRC$ formulas to generic arithmetical formulas. We then interpret these formulas as finite extensions of $T$, so in particular $(\varphi \land \psi)\fsT = \varphi\fsT \land \psi\fsT$, because the union of $T + \varphi\fsT$ and $T + \psi\fsT$ is the same as $T + \varphi\fsT \land \psi\fsT$.

  We define $\cdot\ois$ such that $S(\vec{t})\ois := \tau(u) \lor (u = \gnum{S(\vec{t})\ofs})$, and extend it to non-atomic formulas as described in Definition~\ref{def:isT}.
  Note that $x_k \in \fv(\varphi)$ if and only if $y_k \in \fv(\varphi\oisT)$ and $c_k$ appears in $\varphi$ if and only if $z_k \in \fv(\varphi\oisT)$.

\begin{lemma}
\label{lem:same_thing}

For any strictly positive formula $\varphi$ with free variables $\vec{x}$ and constants $\vec{c}$:
  \begin{equation*}
    T \vdash
      \Forall{\theta} \Forall{\vec y, \vec z} (
        \Box_{\varphi\oisT} \theta \leftrightarrow \Box_T (\varphi\ofsT \to \theta)
      )
    ,
  \end{equation*}
  where $\theta$ is (the Gödel number of) a closed formula.
\end{lemma}

\begin{proof}
  By external induction on $\varphi$. There is nothing to show for $\top$, because since $\tau$ is a bounded formula, $\Box_{\tau}$ is the same as $\Box_T$.

  The case of relational symbols is a straightforward consequence of the formalized Deduction Theorem \cite{Feferman1960}.

In the case of $\land$, we take $\varphi = \psi \land \delta$, and we omit the variables $\vec{y}$ and $\vec{z}$, as they introduce visual clutter but don't make the proof any more complex.

($\to$)
Reason in $T$ and fix an arbitrary $\theta$, assuming $\Box_{\psi\oisT \lor \delta\oisT} \theta$. Then there is a finite sequence $\pi= \pi_0, \ldots ,\pi_n$  with $\pi_n = \theta$ that is a proof of $\theta$ in the theory axiomatized by $\psi\oisT \lor \delta\oisT$. Each formula $\chi_i$ occurring in $\pi$ that is not a consequence of previous formulas in the sequence through a rule satisfies either $\psi\oisT$ or $\delta\oisT$.
Then we have in particular that either $\Box_{\psi\oisT} \chi_i$ or $\Box_{\delta\oisT} \chi_i$ for each such $\chi_i$, whence by the induction hypothesis either $\Box_T(\psi\ofsT \to \chi_i)$ or $\Box_T(\delta\ofsT \to \chi_i)$ holds. In both cases we have $\Box_T(\psi\ofsT \wedge \delta\ofsT \to \chi_i)$, and thus the proof of $\theta$ can be repeated in $T$ under the assumption of $(\psi \land \delta)\ofsT$, as desired.

$(\leftarrow)$
Fix $\theta$, $\vec{y}$ and $\vec z$. By the induction hypothesis (taking $\theta$ to be $\psi\ofsT$) we see that $\Box_{\psi\oisT}\psi\ofsT$ and likewise $\Box_{\delta\oisT} \delta\ofsT$. Thus $\Box_{\psi\oisT \lor \delta\oisT} (\psi\ofsT \land \delta\ofsT)$. By assumption we have $\Box_T(\psi\ofsT \land \delta\ofsT \to \theta)$, and since $\varphi\oisT$ extends $T$, we may conclude $\Box_{\varphi\oisT} \theta$ as desired.

The $\forall$ case follows the same idea as the $\land$ case. Consider $\varphi = \Forall{x_0} \psi$, with $\fv(\psi) = \{x_0, x_1\}$ without loss of generality. Note that $x_0$ is represented by $y_0$ in $T$, and this is always a different variable from any $z$ used to represent $\QRC$ constants. As there is no further complication with constants, we omit them.
Let $\theta$ and $l$ be arbitrary and reason in $T$.

$(\to)$
If $\Box_{\Exists{y_0} \psi\oisT\subst{y_1}{l}} \theta$, then there is a proof $\pi = \pi_0, \hdots, \pi_n$ where $\pi_n = \theta$ and each axiom $\chi_i$ in $\pi$ satisfies $\psi\oisT\subst{y_1}{l}\subst{y_0}{k_i}$ for some number $k_i$, and consequently $\Box_T (\psi\ofsT\subst{y_1}{l}\subst{y_0}{k_i} \to \chi_i)$ by the induction hypothesis for each $i$. Then by weakening we conclude $\Box_T (\Forall{y_0} \psi\ofsT\subst{y_1}{l} \to \chi_i)$ for each $i$, and we are done.

$(\leftarrow)$
Assume $\Box_T (\Forall{y_0} \psi\ofsT\subst{y_1}{l} \to \theta)$. By Lemma~\ref{lem:mod} under the box, we obtain $\Box_T (\Forall{y_0 < m} \psi\ofsT\subst{y_1}{l} \to \theta)$, where $m$ is the size of the domain of $\M$. Using the induction hypothesis for each $k < m$ with $\theta := \psi\ofsT\subst{y_1}{l}\subst{y_0}{k}$, $y_0 := k$, and $y_1 := l$, we get $\Forall{k < m} \Box_{\psi\oisT\subst{y_1}{l}\subst{y_0}{k}} \psi\ofsT\subst{y_1}{l}\subst{y_0}{k}$, and in particular $\Forall{k < m} \Box_{\Exists{y_0} \psi\oisT\subst{y_1}{l}} \psi\ofsT\subst{y_1}{l}\subst{y_0}{k}$. Then by $\Sigma^0_1$ collection we can change the order of the quantifier and the box, concluding $\Box_{\Exists{y_0} \psi\oisT\subst{y_1}{l}} \Forall{y_0 < m} \psi\ofsT\subst{y_1}{l}$, which is enough to conclude this part of the proof.

  Finally, for the case of $\varphi = \Diamond \psi$, we start by observing that applying the induction hypothesis to $\bot$ yields 
$
  T \vdash \Forall{\vec{y}, \vec z} (    
    \Diamond_{\varphi\oisT} \top \leftrightarrow \Diamond_T \varphi\ofsT )
$.
Note that $(\Diamond \psi)\oisT = \tau(u) \lor (u = \gnum{\Diamond_{\psi\oisT} \top})$, and thus $\Box_{(\Diamond \psi)\oisT} \theta$ is equivalent to $\Box_T (\Diamond_{\psi\oisT} \top \to \theta)$ by the formalized Deduction Theorem. The previous observation under the box then suffices to finish the proof.
\end{proof}

We are finally ready to prove arithmetical completeness for any sound c.e.~theory extending $\isig{1}$.

\begin{theorem}[Arithmetical completeness]
  $\QRC \supseteq \QRCt(T)$.
\end{theorem}
\begin{proof}
  Recall the definition of $\QRCt(T)$:
  \begin{equation*} 
  \QRCt(T) = \{\varphi(\vec{x}, \vec{c}) \vdash \psi(\vec{x}, \vec{c}) \mid
    \Forall{\cdot\is} 
      T \vdash \Forall{\theta} \Forall{\vec{y}, \vec{z}} (\Box_{\psi\isT} \theta \to \Box_{\varphi\isT} \theta)
  \}
  ,
  \end{equation*}
where $\theta$ is closed.

  We show that if $\varphi \not\vdash \psi$, then the realization $\cdot\ois$ defined above is such that
  \begin{equation*}
    T \not\vdash \Forall{\theta} \Forall{\vec{y}, \vec{z}} (\Box_{\psi\oisT} \theta \to \Box_{\varphi\oisT} \theta)
    .
  \end{equation*}
  Suppose towards a contradiction that $T$ does prove this formula. Then by Lemma~\ref{lem:same_thing}:
  \begin{equation*}
    T \vdash \Forall{\theta} \Forall{\vec{y}, \vec z} (\Box_T (\psi\ofsT \to \theta) \to \Box_T (\varphi\ofsT \to \theta))
    .
  \end{equation*}
  Taking $\theta := \psi\ofsT\subst{\vec{y}}{\vec{\bij{g_{\varphi, \psi}(x)}}}\subst{\vec z}{\vec{\bij{c^{I_0}}}}$, $\vec{y} := \vec{\bij{g_{\varphi, \psi}(x)}}$, and $\vec z := \vec{\bij{c^{I_0}}}$, we conclude
  \begin{equation*}
    T \vdash
      \Box_T (\varphi\ofsT \to \psi\ofsT)
                \subst{\vec{y}}{\vec{\bij{g_{\varphi, \psi}(x)}}}
                \subst{\vec z}{\vec{\bij{c^{I_0}}}}
    .
  \end{equation*}
  This together with the soundness of $T$ contradicts Theorem~\ref{thm:arithmetical_completeness*}.
\end{proof}

We have seen that $\QRC = \mathcal{QRC}_1(T)$ for any sound c.e.~theory $T$ extending $\isig{1}$. Thus $\mathcal{QRC}_1(T)$ is constant over a large class of theories, and for these theories it does not depend on the specific axiomatization $\tau$ chosen. This is similar to the propositional case, but simpler than the full predicate case, where $\QPL(T)$ is known to depend on both $T$ (as shown by Montagna \cite{Montagna1984}) and $\tau$ (as shown by Artemov \cite{Artemov1986}).

\section{A decidable fragment of \texorpdfstring{$\QPL(\PA)$}{QPL(PA)}}
\label{sec:QPL}

As we mentioned before, Vardanyan's results are very robust and $\Pi^0_2$ completeness can already be obtained for the language with just one unary predicate symbol and no nested occurrences of $\Box$, as shown in \cite{Vardanyan1988}. However, Artemov and Japaridze \cite{ArtemovJaparidze1990} managed to carve out a non-trivial decidable fragment of $\QPL(\PA)$: all formulas that are decidable on finite Kripke frames correspond directly to a fragment of $\QPL(\PA)$.
They further observed that as a corollary one can conclude the decidability of the one variable fragment of $\QPL(\PA)$.

The above may seem contradictory with Vardanyan's result on $\Pi^0_2$ completeness of the fragment of $\QPL(\PA)$ with just one unary predicate symbol. However, although it is easy to see that in predicate logic any sentence with just one unary predicate is equivalent to one in the one-variable fragment, this does not hold when modalities are involved, as exhibited by the formula $\Forall{x} \Forall{y} \Box (P(x) \lor \neg P(y))$.

In this section we shall see that $\QRC$ gives rise to a new decidable fragment of $\QPL(\PA)$. Let $\Lbf$ be the language of full quantified modal logic, based on the same signatures as the language of $\QRC$ and extending the latter with $\to$, $\bot$ and the usual abbreviations. We extend the notion of finite arithmetical realization to $\Lbf$ as follows:
\begin{itemize}
  \item $\bot\fsPA := \bot$;
  \item $(A \to B)\fsPA := A\fsPA \to B\fsPA$.
\end{itemize}

We now define the quantified provability logic of $\PA$ as the set of always provable $\Lbf$ formulas:
\begin{equation*}
  \QPL(\PA) := \{
    A(\vec x, \vec c) \in \Lbf
    \mid
    \text{for any }\cdot\fs, \text{ we have } \PA \vdash \Forall{\vec y, \vec z} A\fsPA
  \}
  ,
\end{equation*}
where $\vec y$ are the arithmetical counterpart of $\vec x$ and $\vec z$ the arithmetical counterpart of $\vec c$, as before.

\begin{theorem}
  Let $\varphi$ and $\psi$ be $\QRC$ formulas. Then
  \begin{equation*}
    \varphi \vdash \psi
    \iff
    \varphi \to \psi \in \QPL(\PA)
    .
  \end{equation*}
\end{theorem}
\begin{proof}
  We begin with the left-to-right implication, which we prove by induction on the proof of $\varphi \vdash \psi$.

  Consider $\varphi \vdash \top$ with $\vec x, \vec c$ the free variables and constants appearing in $\varphi$, respectively. Let $\cdot\fs$ be any realization. We wish to show that $\PA \vdash \Forall{\vec y, \vec z} (\varphi\fsPA \to \top)$, which is clearly the case.
  We are equally easily convinced that $\PA \vdash \Forall{\vec y, \vec z} (\varphi\fsPA \to \varphi\fsPA)$ and that $\PA \vdash \Forall{\vec y, \vec{y'}, \vec z, \vec{z'}} (\varphi\fsPA \land \psi\fsPA \to \varphi\fsPA)$ (where $\vec{x'}$ and $\vec{c'}$ are the variables and constants of $\psi$).

  Let $\vec x$ and $\vec c$ be the constants appearing in either $\varphi$, $\psi$, or $\chi$.
  
  The conjunction introduction rule states that from $\varphi \vdash \psi$ and $\varphi \vdash \chi$ we can obtain $\varphi \vdash \psi \land \chi$. Let $\cdot\fs$ be any realization. We wish to prove $\PA \vdash \Forall{\vec y, \vec z} (\varphi\fsPA \to \psi\fsPA \land \chi\fsPA)$. By the induction hypothesis we know that $\PA \vdash \Forall{\vec y, \vec z} (\varphi\fsPA \to \psi\fsPA)$ and $\PA \vdash \Forall{\vec y, \vec z} (\varphi\fsPA \to \chi\fsPA)$, which easily allows us to prove the desired goal.

  The cut rule states that if $\varphi \vdash \psi$ and $\psi \vdash \chi$, then $\varphi \vdash \chi$. It is enough to show $\PA \vdash \Forall{\vec y, \vec z} (\varphi\fsPA \to \chi\fsPA)$. By the induction hypothesis, we know both $\PA \vdash \Forall{\vec y, \vec z} (\varphi\fsPA \to \psi\fsPA)$ and $\PA \vdash \Forall{\vec y, \vec z} (\psi\fsPA \to \chi\fsPA)$. The result follows handily.

  The necessitation rule states that if $\varphi \vdash \psi$, then $\Diamond \varphi \vdash \Diamond \psi$. We work towards showing $\PA \vdash \Forall{\vec y, \vec z} (\Diamond_\PA \varphi\fsPA \to \Diamond_\PA \psi\fsPA)$.
  By the induction hypothesis, we know that $\PA \vdash \Forall{\vec y, \vec z} (\varphi\fsPA \to \psi\fsPA)$. Taking this under the box, we obtain $\PA \vdash \Box_\PA \Forall{\vec y, \vec z} (\varphi\fsPA \to \psi\fsPA)$. This implies that $\PA \vdash \Forall{\vec y, \vec z} \Box_\PA (\varphi\fsPA \to \psi\fsPA)$ by the formalized converse Barcan Formula, which in turn implies our desired goal.

  Consider now $\Diamond \Diamond \varphi \to \Diamond \varphi$. We wish to show $\PA \vdash \Forall{\vec y, \vec z} (\Diamond_\PA \Diamond_\PA \varphi\fsPA \to \Diamond_\PA \varphi\fsPA)$. This holds by provable $\Sigma^0_1$ completeness.

  We turn to the $\forall$-introduction rule on the right, that if $\varphi \vdash \psi$ then $\varphi \vdash \Forall{x'} \psi$, as long as $x' \not \in \fv(\varphi)$. Let $\vec x$ be the free variables appearing in either $\varphi$ or $\Forall{x'} \psi$. We wish to show $\PA \vdash \Forall{\vec y, \vec z} (\varphi\fsPA \to \Forall{y'} \psi\fsPA)$. By the induction hypothesis, $\PA \vdash \Forall{\vec y, y', \vec z} (\varphi\fsPA \to \psi\fsPA)$. The goal follows from the observation that $y' \not\in \fv(\varphi)$.

  For the $\forall$-introduction rule on the left, that if $\varphi\subst{x'}{t} \vdash \psi$ then $\Forall{x'} \varphi \vdash \psi$ ($t$ free for $x'$ in $\varphi$), we want to show $\PA \vdash \Forall{\vec y, \vec z} (\Forall{y'} \varphi\fsPA \to \psi\fsPA)$.
  Let $w$ be the $\PA$ variable corresponding to $t$. Then $(\varphi\subst{x'}{t})\fsPA$ is $\varphi\fsPA\subst{y'}{w}$ and the induction hypothesis tells us that $\PA \vdash \Forall{\vec y, \vec z, w} (\varphi\fsPA\subst{y'}{w} \to \psi\fsPA)$, which readily implies the desired result.

  For the term instantiation rule, that if $\varphi \vdash \psi$ then $\varphi\subst{x'}{t} \vdash \psi\subst{x'}{t}$, let $w$ be the $\PA$ variable corresponding to $t$. We wish to show $\PA \vdash \Forall{\vec y, \vec z, w} (\varphi\fsPA\subst{y'}{w} \to \psi\fsPA\subst{y'}{w})$. The induction hypothesis is $\PA \vdash \Forall{\vec y, y', \vec z} (\varphi\fsPA \to \psi\fsPA)$, so this is a simple matter of variable renaming.

  Finally, consider the constant elimination rule, that if $\varphi\subst{x'}{c'} \vdash \psi\subst{x'}{c'}$ then $\varphi \vdash \psi$, with $c'$ not appearing in either $\varphi$ nor $\psi$.
  We wish to show $\PA \vdash \Forall{\vec y, y', \vec z} (\varphi\fsPA \to \psi\fsPA)$.
  We know $\PA \vdash \Forall{\vec y, \vec z, z'} (\varphi\fsPA\subst{y'}{z'} \to \psi\fsPA\subst{y'}{z'})$ by the induction hypothesis. The result follows by renaming $z'$ to $y'$ in the induction hypothesis, noting that $z'$ does not appear in either $\varphi\fsPA$ nor $\psi\fsPA$.

  \vspace{3mm}

  We turn now to the right-to-left implication, which we address by taking the contrapositive. Thus, let $\varphi$ and $\psi$ be strictly positive formulas with free variables $\vec{x}$ and constants $\vec{c}$ such that $\varphi \not\vdash \psi$. By Theorem~\ref{thm:arithmetical_completeness*}, we have
  \begin{equation*}
    \PA \not\vdash
      (\varphi\ofsPA \to \psi\ofsPA)
        \subst{\vec{y}}{\vec{\bij{g_{\varphi, \psi}(x)}}}
        \subst{\vec z}{\vec{\bij{c^{I_0}}}}
    ,
  \end{equation*}
and thus $\varphi \to \psi \notin \QPL(\PA)$.

\end{proof}

\section{Heyting Arithmetic}
\label{sec:HA}

We end this paper with a foray into intuitionistic arithmetic. We show that $\QRC$ is sound with respect to Heyting Arithmetic ($\HA$) and conjecture that it is complete as well.

Let $\eta(u)$ be a $\Sigma^0_1$ formula naturally axiomatizing $\HA$.\footnote{$\HA$ has poly-time decidable axiomatizations, but $\Sigma^0_1$ suffices for this section.} The $\HA$-provability of $\varphi$ can thus be expressed by $\Box_\eta \varphi$, where we recall that $\Box_\tau \varphi := \Exists{p} \Prf_\tau(p, \gnum{\varphi})$, where
\begin{align*}
  \Prf_\tau(p, n) :=
    \text{sequence}(p)
    &\land p_{|p|-1} = n \land {}\\
    \Forall{k < |p|} &(
      \tau(p_k)
      \lor \Exists{i, j < k} \text{MP}(p_i, p_j, p_k)
      \lor \Exists{i < k} \text{Gen}(p_i, p_k)
    )
    .
\end{align*}

We observe that a number of standard results in the realm of classical provability logic still hold in the intuitionistic case.

\begin{lemma}[Derivability conditions, \cite{VisserZoethout2019}]
\label{lem:HA_derivability}
  Let $\tau$ be a $\Sigma^0_1$ axiomatization of an arithmetical theory $T$ extending $\HA$, and $\varphi, \psi$ be formulas in the language of arithmetic. Then:
  \begin{enumerate}
    \item if $T \vdash \varphi$ then $\HA \vdash \Box_\tau \varphi$;
    \item $\HA \vdash \Box_\tau (\varphi \to \psi) \to (\Box_\tau \varphi \to \Box_\tau \psi)$;
    \item $\HA \vdash \Box_\tau \varphi \to \Box_\eta \Box_\tau \varphi$.
  \end{enumerate}
\end{lemma}

\begin{lemma}[Collection, \cite{FujiwaraKurahashi2020}, Proposition 5.13]
\label{lem:HAcollection}
  For any arithmetical formula $\varphi$ without $x$ as a free variable:
  \begin{equation*}
    \HA \vdash
    \Forall{y {<} x} \Exists{z} \varphi
      \to
      \Exists{w} \Forall{y {<} x} \Exists{z {<} w} \varphi
      .
  \end{equation*}
\end{lemma}

Note that, since $\HA$ proves full collection, $\Prf_\eta$ is equivalent to a $\Sigma^0_1$ formula, provably in $\HA$.

As in the classical case, we define $\Diamond_\tau \varphi$ as $\neg \Box_\tau \neg \varphi$ when $\tau$ is an axiomatization of an extension of $\HA$.

We extend a generic realization $\cdot\is$ to non-predicate formulas as in the classical case (cf.~Definition~\ref{def:isT}).
Note that $\varphi\isHA$ is equivalent to a $\Sigma_1$ formula and extends $\HA$, both of these provably in $\HA$.

\begin{theorem}[Arithmetical soundness w.r.t.~$\HA$]
\label{thm:HAsoundness}
  \begin{equation*}
    \QRC \subseteq
      \{\varphi \vdash \psi \mid
        \Forall{\cdot\is} \HA \vdash \Forall{\theta, \vec y, \vec z}
          (\Box_{\psi\isHA} \theta \to \Box_{\varphi\isHA} \theta)
      \}
      .
  \end{equation*}
\end{theorem}
\begin{proof}
  Let $\varphi$ and $\psi$ be formulas such that $\varphi \vdash \psi$. The proof is similar to the classical case, and proceeds by external induction on $\varphi \vdash \psi$.

  The soundness of $\varphi \vdash \top$ states that $\varphi\isHA$ extends $\HA$, which can readily be checked by induction on $\varphi$. The soundness of $\varphi \vdash \varphi$ and of the cut rule are straightforward, and the soundness of $\varphi \land \psi \vdash \varphi$ is a simple weakening.

  We proceed with the soundness of the conjunction introduction rule, that if $\varphi \vdash \psi$ and $\varphi \vdash \chi$ then $\varphi \vdash \psi \land \chi$. Reason in $\HA$ and let $\theta$ be a closed formula, and $\vec y$ and $\vec z$ be arbitrary. Assume $\Box_{\psi\isHA \lor \chi\isHA} \theta$, taking $\vec{\pi} = \pi_0, \hdots, \pi_{n-1}$ as a proof of this fact. Some of the $\pi_i$ are axioms of $\psi\isHA$, some are axioms of $\chi\isHA$, and some follow from previous steps in the proof through a rule. Let $\{\delta_i\}_{i \in I}$ be the finite set of $\psi\isHA$ axioms appearing in $\vec{\pi}$. Thus $\Box_{\psi\isHA} \bigwedge_{i \in I} \delta_i$ and by the induction hypothesis for $\varphi \vdash \psi$ we obtain $\Box_{\varphi\isHA} \bigwedge_{i \in I} \delta_i$. On the other hand, we know $\Box_{\chi\isHA} (\bigwedge_{i \in I} \delta_i \to \theta)$ by the deduction theorem. Through the induction hypothesis for $\varphi \vdash \chi$ we conclude $\Box_{\varphi\isHA} (\bigwedge_{i \in I} \delta_i \to \theta)$. Putting these two observations together, we obtain the desired $\Box_{\varphi\isHA} \theta$.

  We turn to the necessitation rule, that if $\varphi \vdash \psi$ then $\Diamond \varphi \vdash \Diamond \psi$. Reason in $\HA$ and let $\theta, \vec y$, and $\vec z$ be arbitrary. Consider the induction hypothesis with $\theta := \neg \top$ (and $\vec y$, $\vec z$ as given by our assumption): $\Box_{\psi\isHA} \neg \top \to \Box_{\varphi\isHA} \neg \top$. Taking the contrapositive, we conclude $\Diamond_{\varphi\isHA} \top \to \Diamond_{\psi\isHA} \top$, and consequently $\Box_\eta (\Diamond_{\varphi\isHA} \top \to \Diamond_{\psi\isHA} \top)$ by Lemma~\ref{lem:HA_derivability}.
  Assume now that $\Box_{(\Diamond \psi)\isHA} \theta$. By the deduction theorem, we obtain $\Box_\eta (\Diamond_{\psi\isHA} \top \to \theta)$. Thus our previous observation gives us $\Box_\eta (\Diamond_{\varphi\isHA} \top \to \theta)$, which is $\Box_{(\Diamond \varphi)\isHA} \theta$ by the deduction theorem.

  Consider the transitivity axiom: $\Diamond \Diamond \varphi \to \Diamond \varphi$. We start by observing that $(\Diamond \Diamond \varphi)\isHA$ is equivalent to $\eta(u) \lor u = \gnum{\Diamond_\eta \Diamond_{\varphi\isHA} \top}$. Note that we can derive $\Box_\eta (\Diamond_\eta \Diamond_{\varphi\isHA} \top \to \Diamond_{\varphi\isHA} \top)$ from Lemma~\ref{lem:HA_derivability}. Assume now $\Box_{(\Diamond \varphi)\isHA} \theta$. By the deduction theorem we have $\Box_\eta (\Diamond_{\varphi\isHA} \top \to \theta)$. Then we obtain $\Box_\eta (\Diamond_\eta \Diamond_{\varphi\isHA} \top \to \theta)$ by our previous observation, and we finish with one more application of the deduction theorem.

  The $\forall$ introduction rule on the right states that if $x_0 \notin \fv(\varphi)$ and $\varphi \vdash \psi$, then $\varphi \vdash \Forall{x_0} \psi$. Reason in $\HA$ and let $\theta, \vec y$, and $\vec z$ be arbitrary, where $y_0$ does not appear in $\vec y$. Assume $\Box_{\Exists{y_0} \psi\isHA} \theta$ and let $\vec{\pi} = \pi_0, \hdots \pi_{n-1}$ be a proof of this fact. Let $\{\delta\}_{i \in I}$ be the finite set of axioms of $\Exists{y_0} \psi\isHA$ appearing in $\vec{\pi}$. Note that $\Box_\eta (\Forall{i \in I} \delta_i \to \theta)$ holds by the deduction theorem.
  Then for each $i \in I$ there is $k_i$ such that $\delta_i$ is an axiom of $\psi\isHA\subst{y_0}{k_i}$, so in particular $\Box_{\psi\isHA\subst{y_0}{k_i}} \delta_i$. We can use the induction hypothesis for each $i \in I$ to conclude $\Box_{\varphi\isHA\subst{y_0}{k_i}} \delta_i$, and since $x_0 \not\in \fv(\varphi)$, we also know that $y_0 \not\in \fv(\varphi\isHA)$, and thus we obtain $\Forall{i \in I} \Box_{\varphi\isHA} \delta_i$. We now use collection to obtain $\Box_{\varphi\isHA} \Forall{i \in I} \delta_i$, and the result follows from our observation that $\Box_\eta (\Forall{i \in I} \delta_i \to \theta)$, noting that $\varphi\isHA$ extends $\HA$.

  The $\forall$ introduction on the left states that if $\varphi\subst{x_0}{t} \vdash \psi$, then $\Forall{x_0} \varphi \vdash \psi$. Let $w$ be the arithmetical counterpart of $t$ (so if $t$ is $x_k$ then $w := y_k$ and if $t$ is $c_k$ then $w := z_k$). We have $\varphi\subst{x_0}{t}\isHA = \varphi\isHA\subst{y_0}{w}$. Let $\theta, \vec y$ and $\vec z$ be arbitrary, where $y_0$ appears in $\vec y$ if and only if $x_0$ is a free variable of $\psi$. We assume $\Box_{\psi\isHA} \theta$. If $t$ appears in $\varphi$ or in $\psi$, then the value of $w$ was already fixed when we picked arbitrary $\vec y$ and $\vec z$. Otherwise, fix $w := 0$ and use the induction hypothesis to obtain $\Box_{\varphi\isHA\subst{y_0}{w}} \theta$. It is then clear that $\Box_{\Exists{y_0} \varphi\isHA} \theta$ holds as well.

  The term instantiation rule states that if $\varphi \vdash \psi$, then $\varphi\subst{x_0}{t} \vdash \psi\subst{x_0}{t}$. Let $w$ be the arithmetical counterpart of $t$. The induction hypothesis states that $\HA \vdash \Forall{\theta, y_0, \vec y, \vec z} (\Box_{\psi\isHA} \theta \to \Box_{\varphi\isHA} \theta)$ (assuming without loss of generality that $y_0$ is a free variable of either $\varphi$ or $\psi$). We wish to prove that $\HA \vdash \Forall{\theta, w, \vec y, \vec z} (\Box_{\psi\isHA\subst{y_0}{w}} \theta \to \Box_{\varphi\isHA\subst{y_0}{w}} \theta)$. This is a simple matter of renaming variables.

  The constant elimination rule states that if $c_0$ does not appear in either $\varphi$ nor $\psi$ and $\varphi\subst{x_0}{c_0} \vdash \psi\subst{x_0}{c_0}$, then $\varphi \vdash \psi$. The induction hypothesis states that $\HA \vdash \Forall{\theta, \vec y, z_0, \vec z} (\Box_{\psi\isHA\subst{y_0}{z_0}} \theta \to \Box_{\varphi\isHA\subst{y_0}{z_0}} \theta)$, where $y_0$ is not a part of $\vec y$. We can rename $z_0$ back to $y_0$ to obtain $\HA \vdash \Forall{\theta, y_0, \vec y, \vec z} (\Box_{\psi\isHA} \theta \to \Box_{\varphi\isHA} \theta)$, noting that $\varphi\isHA\subst{y_0}{z_0}\subst{z_0}{y_0}$ is just $\varphi\isHA$ because $z_0$ does not appear in $\varphi\isHA$ (and similarly for $\psi$). This concludes the proof.
\end{proof}

The arithmetical completeness of $\QRC$ with respect to $\HA$ remains an open question.
We conjecture that the fact that the substitutions in the completeness proofs are of restricted complexity, and the fact that $\PA$ is $\Pi^0_2$ conservative over $\HA$ (see \cite{Friedman1978}) can be essential ingredients in a possible completeness proof.

\section{Future work}
\label{sec:future}

There are many unexplored paths surrounding $\QRC$. It would be worthwhile to extend it to a polymodal language (in analogy with $\RC$, as in \cite{Beklemishev2012}), and to the positive setting (as in \cite{Dunn1995, CelaniJansana2012}). Whether this is possible without loosing decidability remains to be seen.
A hypothetical $\mathsf{QRC}_\Lambda$ would presumably lead to some interesting applications to $\Pi^0_1$ ordinal analysis and ordinal notation systems.

There are several proof theoretical properties of interest, such as interpolation, cut-free proofs, fixpoints, etc.
One could also strive for uniform completeness.

Moreover, the set of always true $\QRC$ sequents should be a productive avenue of study.

The completeness of $\QRC$ with respect to $\HA$ remains an open question.

Finally, the current approximation for the computational complexity of $\QRC$ is super-exponential space, since the canonical model grows quite large. A more dedicated study might lead to a significant reduction in complexity.




\newcommand{\noopsort}[1]{}

\end{document}